\documentclass[11pt,leqno,oneside,letterpaper]{amsart}
\usepackage[letterpaper,left=1.75truein, top=1truein]{geometry}
\usepackage{amsmath,amsfonts,amssymb,amscd,amsthm,amsbsy,epsf,graphics}

\usepackage{color}

\usepackage[colorlinks,linkcolor=blue,citecolor=blue]{hyperref} 
\usepackage{epsfig} 
\usepackage{tikz-cd}
\usepackage{bm}

\usepackage{graphicx}
\DeclareGraphicsExtensions{.pdf}

\newtheorem{thm}{Theorem}[section]
\newtheorem{cor}[thm]{Corollary}
\newtheorem{lemma}[thm]{Lemma}
\newtheorem{prop}[thm]{Proposition}
\newtheorem{conj}[thm]{Conjecture}

\newtheorem{que}[thm]{Question}
\newtheorem{rem}[thm]{Remark}

\newtheorem{examples}[thm]{Examples}
\newtheorem{problem}[thm]{Problem}
\newtheorem{question}[thm]{Question} 
\newtheorem{remark}[thm]{Remark} 
\newtheorem{remarks}[thm]{Remarks}

\theoremstyle{definition}

\newtheoremstyle{cases}
  {12pt plus 6 pt}
  {2pt}
  {\bfseries}   
  {}
  {\bfseries}
  {.}
  {.5em}
  {}
\theoremstyle{cases}

\numberwithin{subcase}{case} \numberwithin{subsubcase}{subcase}
\numberwithin{equation}{subsection} 

\begin{document}

\title{Branched covers of quasipositive links and L-spaces}

\author{Michel Boileau}  
\thanks{Michel Boileau was partially supported by ANR projects 12-BS01-0003-01 and 12-BS01-0004-01.}
\address{Aix Marseille Univ, CNRS, Centrale Marseille, I2M, Marseille, France, 39, rue F. Joliot Curie, 13453 Marseille Cedex 13}
\email{michel.boileau@cmi.univ-mrs.fr}
 
\author[Steven Boyer]{Steven Boyer}
\thanks{Steven Boyer was partially supported by NSERC grant RGPIN 9446-2013}
\address{D\'epartement de Math\'ematiques, Universit\'e du Qu\'ebec \`a Montr\'eal, 201 avenue du Pr\'esident-Kennedy, Montr\'eal, QC H2X 3Y7.}
\email{boyer.steven@uqam.ca}
\urladdr{http://www.cirget.uqam.ca/boyer/boyer.html}

\author[Cameron McA. Gordon]{Cameron McA. Gordon}
\thanks{Cameron Gordon was partially supported by NSF grant DMS-130902.}
\address{Department of Mathematics, University of Texas at Austin, 1 University Station, Austin, TX 78712, USA.}
\email{gordon@math.utexas.edu}
\urladdr{http://www.ma.utexas.edu/text/webpages/gordon.html}

\begin{abstract} 
Let $L$ be a oriented link such that $\Sigma_n(L)$, the $n$-fold cyclic cover of $S^3$ branched over $L$, is an L-space for some $n \geq 2$. We show that if either $L$ is a strongly quasipositive link other than one with Alexander polynomial a multiple of $(t-1)^{2g(L) + (|L|-1)}$, or $L$ is a quasipositive link other than one with Alexander polynomial divisible by $(t-1)^{2g_4(L) + (|L|-1)}$, then there is an integer $n(L)$, determined by the Alexander polynomial of $L$ in the first case and the Alexander polynomial of $L$ and the smooth $4$-genus of $L$, $g_4(L)$,  in the second, such that $n \leq n(L)$. If $K$ is a strongly quasipositive knot with monic Alexander polynomial such as an L-space knot, we show that $\Sigma_n(K)$ is not an L-space for $n \geq 6$, and that the Alexander polynomial of $K$ is a non-trivial product of cyclotomic polynomials if $\Sigma_n(K)$ is an L-space for some $n = 2, 3, 4, 5$. Our results allow us to calculate the smooth and topological 4-ball genera of, for instance, quasi-alternating quasipositive links. They also allow us to classify strongly quasipositive alternating links and $3$-strand pretzel links. 
\end{abstract}

\maketitle

\begin{center}
\today 
\end{center}

\section{Introduction} \label{sec: introduction} 
We assume throughout the paper, unless otherwise stated, that links are oriented and contained in the $3$-sphere. Knots are assumed to be non-trivial. To each link $L$ and integer $n \geq 2$, we associate an $n$-fold cyclic cover $\Sigma_n(L) \to S^3$ branched over $L$ (\S \ref{sec: signature function}). 
Set 
$$\mathcal{L}_{br}(L) = \{n \geq 2: \Sigma_n(L) \hbox{ is an L-space}\}$$
Consider the case of knots $K$. Existing results suggest that if $n \in \mathcal{L}_{br}(K)$ and $2 \le r < n$ then $r \in \mathcal{L}_{br}(K)$, so $\mathcal{L}_{br}(K)$ is either $\emptyset$, or $\{n : 2 \le n \}$, or $\{n : 2 \le n \le N \}$ for some $N \ge 2$. Each possibility is known to occur. Few general results in the area are known, though there are results on various families of examples (see e.g. \cite{GLid1}). One instance which can be handled {\it stably} (i.e. when $n \gg 0$) occurs when $K$ is an {\it L-space knot}, that is, a knot which admits a non-trivial L-space surgery. In this case, work of Roberts (\cite{Ro}; see \cite[Theorem 4.1]{HKM2}) can be used to show that for $n \gg 0$, $\Sigma_n(K)$ admits a co-oriented taut foliation and hence cannot be an L-space (\cite{Bn}, \cite{KR}). Hedden and Mark used Heegaard Floer calculations to obtain similar conclusions (\cite[Corollary 5]{HM}). Corollary \ref{cor: lspace knots} of  our main theorem, Theorem \ref{thm: sqp links}, provides a considerable sharpening of these results by showing that when $K$ is an L-space knot such that $\Sigma_n(K)$ is an L-space for some $n$, then $n \leq 5$ and if $n \geq 4$ then $K$ is the trefoil. Before stating Theorem \ref{thm: sqp links}, we fix some notation. 

Recall that the (reduced) Alexander polynomial of a link $L$, denoted $\Delta_L(t)$, is a $\pm$-palindromic element of the ring $\mathbb Z[t, t^{-1}]$, well-defined up to multiplication by arbitrary signed powers $\pm t^n$ (see \S \ref{sec: signature function}). 

We denote the Tristram-Levine signature function of $L$ by $\sigma_L: S^1 \to \mathbb Z$. Its value at $-1$ is the classical Murasugi signature of $L$: $\sigma(L) = \sigma_L(-1)$. 

Let $B^4 = \{x \in \mathbb R^4 : \|x\| \leq 1\}$ be the standard closed $4$-ball. The $3$-sphere genus, smooth $4$-ball genus, and (topologically) locally flat $4$-ball genus of $L$ will be denoted, respectively, by $g(L)$, $g_4(L)$, and $g_4^{top}(L)$. 

We say that a link $L$ of $m$ components is 
{\it definite} if $\vert \sigma(L) \vert = 2g(L) + (m-1)$ and is {\it indefinite} otherwise.

Let $\Phi_k(t) \in \mathbb Z[t]$ be the $k^{th}$ cyclotomic polynomial.

For each $\zeta \in S^1$ we define $I_-(\zeta)$, respectively $I_+(\zeta)$, to be the open subarc of the circle with endpoints $\zeta, \bar \zeta$ which contains $-1$, respectively $+1$ where we take $I_+(1) = I_-(-1) = \emptyset$.  

Similarly we define $\bar{I}_-(\zeta)$, respectively $\bar{I}_+(\zeta)$, to be the closed subarc of the circle with endpoints $\zeta, \bar \zeta$  which contains $-1$, respectively $+1$. 

Set
$$\zeta_n = \exp(2 \pi i / n)$$

\begin{thm}  \label{thm: sqp links} 
Suppose that $L$ is a strongly quasipositive link\footnote{The various notions of positivity that arise in the paper are defined in \S \ref{sec: positivity}. Since our results hold for a link $L$ if and only if they hold for its mirror image, we take the convention that $L$ is a positive braid link, a positive link, a strongly quasipositive link, or a quasipositive link if either $L$ or its mirror image has this property.} of $m$ components such that $\Sigma_n(L)$ is an L-space for some $n \geq 2$.

$(1)$ All the roots of $\Delta_L(t)$ are contained in $I_+(\zeta_n)$ and $|\sigma_{L}(\zeta)| = 2g(L) + (m-1) = \hbox{{\rm deg}}( \Delta_L(t))$ for $\zeta \in \bar{I}_-(\zeta_n)$. In particular, $L$ is definite.

$(2)$ $g_4^{top}(L) = g(L)$. Further, any locally flat, compact, oriented surface $F$ properly embedded in $B^4$ with oriented boundary $L$ which realises $g_4^{top}(L)$ is connected. 

$(3)$ If $\Delta_L(t)$ is not a non-zero multiple of $(t-1)^{2g(L) + (m-1)}$, there is a positive integer $n_3(L)$ determined by $\Delta_L(t)$ such that $n \leq n_3(L)$. 

$(4)$  If $\Delta_L(t)$ is monic but not $(t-1)^{2g(L) + (m-1)}$, then $n \leq 5$. And if 

$(a)$ $n = 2$, $\Delta_L(t)$ is a non-trivial product of cyclotomic polynomials; 

$(b)$ $n = 3$, $\Delta_L(t)$ is a non-trivial product of powers of $\Phi_1, \Phi_4, \Phi_6$ and $\Phi_{10}$; 

$(c)$ $n \in \{4, 5\}$, $\Delta_L(t)$ is a non-trivial product of powers of $\Phi_1$ and $\Phi_{6}$.

\end{thm}
The exclusion of  the case where $\Delta_L(t)$ is a power of $t-1$ in part (3) of Theorem \ref{thm: sqp links} is necessary as the Hopf link $L$ is strongly quasipositive, has Alexander polynomial $t-1$, and $\Sigma_n(L)$ is the lens space $L(n,1)$ for $n \geq 1$.

The following corollary of Theorem \ref{thm: sqp links} deals with the case that $L$ is a strongly quasipositive knot with monic Alexander polynomial, a case of particular interest, as we will discuss below. 

\begin{cor} \label{cor: monic sqp knots}
Suppose that $K$ is a strongly quasipositive knot with monic Alexander polynomial. 

$(1)$ $\Sigma_n(K)$ is not an L-space for $n \geq 6$. 

$(2)$ If $\Sigma_n(K)$ is an L-space for some $2 \leq n \leq 5$, then $K$ is definite. Moreover,  if 

$(a)$ $n = 2$, then $\Delta_K(t)$ is a non-trivial product of cyclotomic polynomials; 

$(b)$ $n = 3$, then $\Delta_K(t)$ is a non-trivial product of powers of $\Phi_6$ and of $\Phi_{10}$; 

$(c)$ $n \in \{4, 5\}$, then $\Delta_K(t)$ is a non-trivial power of $\Phi_{6}$.

\end{cor}

Corollary \ref{cor: monic sqp knots} is sharp in that if $K$ is a torus knot, it is a strongly quasipositive knot with monic Alexander polynomial and $\Sigma_n(K)$ is an L-space if and only if 
\vspace{-.2cm} 
\begin{itemize}

\item $n = 2$ and $K$ is the $(2,k)$, $(3,4)$, or $(3,5)$ torus knot. In each case, $K$ is definite and $\Delta_K(t)$ is a non-trivial product of cyclotomics; 

\vspace{.2cm} \item $n = 3$ and $K$ is a $(2,3)$ or $(2,5)$ torus knot. In the first case, $K$ is definite and $\Delta_K(t) = \Phi_{6}$ while in the the second case, $K$ is definite and $\Delta_K(t) = \Phi_{10}(t)$; 

\vspace{.2cm} \item $n \in \{4, 5\}$ and $K$ is a $(2,3)$ torus knot. In this case, $K$ is definite and $\Delta_K(t) = \Phi_{6}$.

\end{itemize}
\vspace{-.2cm} 
See \cite[Theorem 1.2]{GLid1}. 

One of the questions which motivated this study is due to Allison Moore.  

\begin{question} {\rm (Allison Moore)} \label{qu: moore}
{\rm If $K$ is a hyperbolic L-space knot, is it true that $\Sigma_2(K)$ is not an L-space?}
\end{question}
\vspace{-.3cm} 
Torus knots are L-space knots and like torus knots, L-space knots are known to be strongly quasipositive (\cite[Theorem 1.2]{He}) and fibred (\cite[Corollary 1.3]{Ni}). We show in Corollary \ref{cor: lspace knot satellite} that no branched cover of a satellite L-space knot is an L-space, but it is unknown whether the same is true for hyperbolic L-space knots.

\begin{cor} \label{cor: lspace knots}
If $K$ is an L-space knot such that $\Sigma_n(K)$ is an L-space for some $n \geq 2$, then $n \leq 5$. Moreover, 

$(1)$ if $n \in \{4, 5\}$, then $K$ is the $(2,3)$ torus knot; 

$(2)$ if $n = 3$, then $K$ is either the $(2,3)$ or $(2,5)$ torus knot, or it is a hyperbolic knot with $\Delta_K(t) = \Phi_{10}(t)$.  

\end{cor}
We expect that $K$ is the $(2,3)$ or $(2,5)$ torus knot in the case $n = 3$ of the corollary. It is interesting to note that this conclusion follows from the Ozsv\'ath-Szab\'o Poincar\'e conjecture: {\it The only irreducible $\mathbb Z$-homology $3$-sphere L-spaces are $S^3$ and the Poincar\'e homology $3$-sphere.} Indeed, if $K$ is a prime knot for which $\Delta_K(t) = \Phi_{10}(t)$, then $\Sigma_3(K)$ is an irreducible $\mathbb Z$-homology $3$-sphere. If it is also an L-space, the Ozsv\'ath-Szab\'o conjecture combines with the orbifold theorem to show that  $K$ is the $(2,5)$ torus knot. We remark that Filip Misev has constructed an infinite family of hyperbolic, fibred, definite, strongly quasipositive knots with Alexander polynomial $\Phi_{10}(t)$ (\cite{Mis}), though he informs us that he and Gilberto Spano have shown that none of them are L-space knots.
 
Stronger results can be obtained for some other families of fibred strongly quasipositive knots. 

\begin{cor} \label{cor: positive braid or divide knot}
Suppose that $K$ is prime and either a positive braid knot, a divide knot, or a fibred strongly quasipositive knot which is either alternating or Montesinos. Then some $\Sigma_n(K)$ is an L-space if and only if $K$ is either a $(2,k)$ torus knot, the $(3,4)$ torus knot, or the $(3,5)$ torus knot. 
\end{cor}
\vspace{-.3cm} 
We derive a similar conclusion for arborescent knots which bound surfaces obtained by plumbing positive Hopf bands along a tree, and so are fibred and strongly quasipositive\footnote{We adopt the usual convention that a positive Hopf band has boundary a positive Hopf link.}.  See Proposition \ref{prop: arborescent}.

A knot is called {\it homologically thin}, or H-thin, for short, if its reduced Khovanov homology is supported on a line. Otherwise it is called H-thick. It is known that alternating knots, and in particular $(2,k)$ torus knots, are H-thin (\cite{Lee}). On the other hand, Khovanov observed that the $(3,4)$ and $(3,5)$ torus knots are H-thick (\cite[\S 6.2]{Kh}) and asked if the $(2,k)$ torus knots are the only prime positive braid knots which are H-thin (\cite[Problem 6.2]{Kh}). It follows from \cite[Theorem 1.1]{OS2} that the $2$-fold branched covers of H-thin knots are L-spaces, so we can answer Khovanov's question through the use of Corollary \ref{cor: positive braid or divide knot}. 

\begin{cor}
Prime positive braid knots other than the $(2,k)$ torus knots are H-thick. 
\qed
\end{cor}   

Call an oriented link $L$ {\it simply laced arborescent} if it is the boundary of the surface obtained by plumbing positive Hopf bands according to one of the trees $\Gamma$ determined by the simply laced Dynkin diagrams $\Gamma = A_m, D_m, E_6, E_7, E_8$. These are precisely the trees whose associated quadratic forms are negative definite. See \cite[pages 61-62]{HNK} for example. We denote $L$ by $L(\Gamma)$. Then,  

\indent \hspace{5.5cm} $L(A_m) = T(2, m+1)$

\indent \hspace{5.5cm} $L(D_m) = P(-2, 2, m-2)$

\indent \hspace{5.5cm} $L(E_6) = P(-2,3,3) = T(3,4)$

\indent \hspace{5.5cm} $L(E_7) = P(-2,3,4)$

\indent \hspace{5.5cm} $L(E_8) = P(-2,3,5) = T(3,5)$.

Given Corollary \ref{cor: positive braid or divide knot} we propose the following conjecture which generalizes Question \ref{qu: moore}. 

\begin{conj} \label{conj: branched lspace implies simply laced arborescent}
{\rm  If $L$ is a prime, fibred, strongly quasipositive link for which some $\Sigma_n(L)$ is an L-space, then $L$ is simply laced arborescent. In case $L$ is a knot $K$, then $K$ would be a $(2,k)$, $(3,4)$, or $(3,5)$ torus knot.}
\end{conj}
\vspace{-.3cm} 
 In another paper, \cite{BBG}, we study the conjecture in the context of basket links, which form a generic subfamily of the set of fibred, strongly quasipositive links. 

With regard to this discussion, it is interesting to consider the following conjecture of Li and Ni (\cite[Conjecture 1.3]{LN}): {\it If $K$  is an L-space knot and each root of its Alexander polynomial lies  on the unit circle, then K is an iterated torus knot.} If this conjecture holds, Corollary \ref{cor: monic sqp knots} implies that the only L-space knots for which some $\Sigma_n(K)$ can be an L-space are iterated torus knots, and in this case Gordon and Lidman have shown that $K$ is a torus knot (\cite{GLid1, GLid2}). For strongly quasipositive iterated torus knots, the same conclusion follows from our results on strongly quasipositive satellite knots. See Proposition \ref{prop: satellite}. 

It is known that the $3$-sphere and smooth $4$-ball genera of strongly quasipositive links coincide (cf. \S \ref{sec: positivity}). Since $|\sigma(L)| \leq 2g_4^{top}(L) + (m-1)$ if some $\Sigma_n(L)$ is a rational homology $3$-sphere (Proposition \ref{prop: florens inequality qhs 2}), this fact can be strengthened if we add the condition that some branched cover of the knot is an L-space. 

\begin{cor} \label{cor: qa and sqp links}
Suppose that $L$ is a strongly quasipositive link of $m$ components for which some $\Sigma_n(L)$ is an L-space.  Then $g_4^{top}(L) = g(L) = \frac12(|\sigma(L)| - (m-1))$ and $\hbox{{\rm deg}}( \Delta_L(t)) = |\sigma(L)|$. In particular, this holds for quasi-alternating strongly quasipositive links. 
\qed
\end{cor}

\begin{examples} \label{examples: sqp}

{\rm (1) If $K$ is an indefinite strongly quasipositive knot, Corollary \ref{cor: qa and sqp links} implies that no $\Sigma_n(K)$ is an L-space. This applies to $41$ of the $251$ prime knots of crossing number $12$ or less that KnotInfo lists as strongly quasipositive.  

(2) Corollary \ref{cor: qa and sqp links} also provides an obstruction to strong quasipositivity for knots with L-space branched covers. For instance, there are $27$ indefinite quasi-alternating knots amongst the $42$ knots of 12 or fewer crossings that KnotInfo lists as having unknown strong quasipositivity status: {\small $11n17, 11n91, 11n99, 11n113, 11n162, 12n171, 12n176, 12n247, 12n270, 12n383, 12n441, 12n496, $ \\ 
$ 12n520, 12n564, 12n626, 12n698, 12n699, 12n700, 12n701, 12n726, 12n734,12n735, 12n796, 12n797, 12n814, \\ 12n863$, and $12n867$.  }
Corollary \ref{cor: qa and sqp links} implies that none of these knots are strongly quasipositive.  
}
\end{examples}

Corollary \ref{cor: qa and sqp links} allows us to characterize which $3$-strand pretzel links are strongly quasipositive. See Proposition \ref{prop:3-strand links}. Here is the statement for $3$-strand pretzel knots. 

\begin{cor} \label{cor: 3-strand knots}
A $3$-strand pretzel knot $K = P(p, q, r)$ with $p, q, |r| \geq 2$ is strongly quasipositive if and only if $(p, q, r)$ verifies one of the following conditions: 

$(1)$ $r > 0$ and  $p, q , r$ are odd; 

$(2)$ $r <0$ and $r$ is even; 

$(3)$ $r <0$, $p, q, r$ are odd and $|r| < min(p, q)$.
 
\end{cor}

Some of our results extend to quasipositive links, though with weaker conclusions.

\begin{thm}  
\label{thm: qp links}
Suppose that $L$ is a quasipositive link of $m$ components such that $\Sigma_n(L)$ is an L-space for some $n \geq 2$.

$(1)$ $g_4^{top}(L) = g_4(L) = \frac12(|\sigma_{L}(\zeta_n^j)| - (m-1))$ for $1 \leq j \leq n-1$. Further, any locally flat, compact, oriented surface $F$ properly embedded in $B^4$ with oriented boundary $L$ which realises $g_4^{top}(L)$ is connected.

$(2)$ If $\Delta_L(t)$ is not divisible by $(t-1)^{2g_4(L) + (m-1)}$\footnote{As a link with $m$ components, $\Delta_{L}(t)$ is divisible by $(t-1)^{m-1}$, so the assumption that $\Delta_{L}(t)$ is not divisible by $(t-1)^{2g_4(L) + (m-1)}$ implies that $g_4(L) > 0$.}, there is a positive integer $n_4(L)$ determined by $\Delta_L(t)$ and $g_4(L)$ such that $n \leq n_4(L)$. 

$(3)$ If $g(L) = g_4(L)$, then the conclusions of Theorem \ref{thm: sqp links} hold.   
\end{thm}

With regard to Theorem \ref{thm: qp links}(3), Hedden has shown that a fibred quasipositive knot $K$ for which $g(K) = g_4(K)$ is strongly quasipositive (\cite[Corollary 1.6]{He}) and Baader has asked more generally whether a quasipositive knot $K$ for which $g(K) = g_4(K)$ is strongly quasipositive (\cite[Question 3, page 268]{Baa1}). We prove in \S \ref{sec: sqp alt} that a quasipositive alternating link $L$ is strongly quasipositive if and only if $g(L) = g_4(L)$.

\begin{cor} \label{cor: qa and qp links}
Suppose that $L$ is a quasipositive link of $m$ components for which some $\Sigma_{2n}(L)$ is an L-space. Then $g_4^{top}(L) = g_4(L) = \frac12(|\sigma(L)| - (m-1))$. In particular, this holds for quasi-alternating quasipositive links. 
\qed
\end{cor}

Corollary \ref{cor: qa and qp links} can be used to calculate the $4$-ball genera of quasipositive links which are, for instance, quasi-alternating (see Example \ref{examples: qp}(1)), or more generally H-thin (cf. \cite{Kh}).

It has been conjectured that for closed, connected, orientable, irreducible $3$-manifolds, the conditions of not being an L-space (NLS), of having a left-orderable fundamental group (LO), and of admitting a co-oriented taut foliation (CTF) are equivalent. (See Conjecture 1 of \cite{BGW} and Conjecture 5 of \cite{Ju}.) Thus it is natural to ask whether the analogues of Theorems \ref{thm: sqp links} and \ref{thm: qp links} hold with the condition NLS replaced by either CTF or LO. We discuss this topic further in \S \ref{sec: questions and problems}. 

Here is how the paper is organised. In \S \ref{sec: signature function} we describe the basic properties and relations between the signature functions of Tristram-Levine and Milnor. In \S \ref{sec: mt inequality} we discuss the relation between the Tristram-Levine signature function of a link $L$ and the Hermitian intersection forms defined on the homology of certain cyclic branched covers of the four-ball. This discussion culminates in a proof of the Murasugi-Tristram inequality for the locally flat case (Theorem \ref{prop: mt inequality}), a folklore result. Section \ref{sec: genera} discusses various genera associated to links and a result of Rudolph (Proposition \ref{prop: rudolph}) describing a class of surfaces which minimizes them. In \S \ref{sec: positivity} we define various notions of positivity for links and describe how they relate. The proofs of Theorem \ref{thm: sqp links} and Corollaries \ref{cor: monic sqp knots} and \ref{cor: lspace knots} are dealt with in \S \ref{sec: sqp}. We also prove Proposition \ref{prop: satellite}, respectively Corollary \ref{cor: monic sqp knots}, which examine when a strongly quasipositive satellite knot, respectively  satellite L-space knot, can have an L-space branched cyclic cover. In \S \ref{sec: sqp alt} we show that the class of strongly quasipositive alternating links coincides with the class of special alternating links. See Proposition \ref{prop: definite alternating}. Section \ref{sec: sqp pretzels} applies the results of \S \ref{sec: sqp} to obstruct a pretzel link from being strongly quasipositive. In particular we are able to determine precisely which $3$-strand pretzel links are strongly quasipositive. Section \ref{sec: fsqpm knots} examines several families of fibred strongly quasipositive knots culminating in the proof of Corollary \ref{cor: positive braid or divide knot}, while \S \ref{sec: nfsqpk} calculates the constant $n_3(K)$ for various families of non-fibred strongly quasipositive knots. The extension of Theorem \ref{thm: sqp links} to quasipositive links is discussed in \S \ref{sec: qp links}. Theorem \ref{thm: qp links} is proved here and various families of examples are examined. Finally, in \S \ref{sec: questions and problems} we discuss questions and open problems which arise from the results of this paper. 

{\bf Acknowledgements}. This paper is a tale of four institutes. It originated during the authors' visit to the {\it Advanced School on Geometric Group Theory and Low-Dimensional Topology} held 23 May - 3 June 2016 at the ICTP Trieste. It was further developed during their visits to the winter-spring 2017 thematic semester {\it Homology theories in low-dimensional topology} held at the Isaac Newton Institute for the Mathematical Sciences Cambridge (funded by EPSRC grant no. EP/K032208/1) and to the CNRS UMI-CRM Montreal during April 2017. It was completed at the Casa Mat\'ematica Oaxaca during the BIRS-CMO workshop {\it Thirty Years of Floer Theory for 3-Manifolds}, July 30 -- August 4, 2017. The authors gratefully acknowledge their debt to these institutes. They also thank Idrissa Ba, who produced the paper's figures, and an anonymous referee who pointed out an error in our initial definition of the integer $n_4(L)$ of Theorem \ref{thm: qp links}(2) and whose meticulous reading of the paper led to an improved exposition.

\section{Signatures, nullities and Alexander polynomials of links} \label{sec: signature function} 

In this section, $L \subset S^3$ will denote a link of $m$ components with exterior $M$. There is an infinite cyclic cover $\widetilde{M} \to M$ associated to the epimorphism $\pi_1(M) \to \mathbb Z$ which sends each of the meridians of $L$, oriented positively with respect to the orientation of $L$, to $1$. Reducing (mod $n$) produces an $n$-fold cyclic cover $M_n \to M$ ($2 \leq n < \infty$). Let $\Sigma_n(L)$ denote the $n$-fold cyclic cover of $S^3$ branched over $L$ determined by $M_n \to M$. 

A {\it Seifert surface} for $L$ is a compact oriented surface with no closed components whose oriented boundary is $L$. Each Seifert surface $F$ of $L$ determines a bilinear Seifert form $\mathcal{S}_F: H_1(F) \times H_1(F) \to \mathbb Z$ for which $\mathcal{S}_F - \mathcal{S}_F^T$ is the intersection form on $H_1(F)$. The {\it Alexander polynomial of $L$} is defined to be the element $\Delta_L(t)$ of $\mathbb Z[t, t^{-1}]$, well-defined up to multiplication by units $\pm t^k$, represented by $\det(\mathcal{S}_F - t\mathcal{S}_F^T)$. 

For each $\zeta \in S^1$, $\mathcal{S}_F(\zeta) = (1-\zeta)\mathcal{S}_F + (1-\bar \zeta)\mathcal{S}_F^T$ defines a Hermitian form on $H_1(F)$ whose signature and nullity were shown by Tristram \cite{Tri} and Levine \cite{Lev} to be independent of the Seifert surface $F$ chosen for $L$. The {\it Tristram-Levine signature function of $L$} is defined by
$$\sigma_L: S^1 \to \mathbb Z, \; \sigma_L(\zeta) = \hbox{signature}(\mathcal{S}_F(\zeta)),$$
while the {\it nullity function of $L$} is defined by
$$\eta_L: S^1 \to \mathbb Z, \; \eta_L(\zeta) = \hbox{nullity}(\mathcal{S}_F(\zeta)).$$
The value of $\sigma_L$ at $\zeta = -1$ is the classical Murasugi signature of $L$, which we denote by $\sigma(L)$. 

Here is a list of some well-known properties of  $\sigma_L, \eta_L$, and $\Delta_L$. 
\vspace{-.2cm}
\begin{itemize}

\item $\sigma_L(\zeta) = \sigma_L(\bar \zeta)$ and $\eta_L(\zeta) = \eta_L(\bar \zeta)$ for all $\zeta$;

\vspace{.2cm} \item $\sigma_L$ and $\eta_L$ are constant on the components of $S^1 \setminus \Delta_L^{-1}(0)$;

\vspace{.2cm} \item $\eta_L(\zeta) \leq m-1$ for $\zeta \in S^1 \setminus \Delta_L^{-1}(0)$; 

\vspace{.2cm} \item $|H_1(\Sigma_n(L))| = \prod_{j=1}^{n-1} |\Delta_L(\zeta_n^j)|$ ; 

\vspace{.2cm} \item $\beta_1(\Sigma_n(L)) = \sum_{j=1}^{n-1} \eta_L(\zeta_n^j)$; 

\vspace{.2cm} \item if $L$ has a Seifert surface with $\mu$ components, then 
$\beta_1(\Sigma_n(L)) \geq (n-1)(\mu-1)$. 

\end{itemize}
\vspace{-.2cm} 

The first property follows from the definition of $\mathcal{S}_F(\zeta)$ while the second is a consequence of the identity $\det(\mathcal{S}_F(\zeta)) = \zeta^{m} (1 - \zeta)^{\beta_1(F)}\Delta_L(\zeta)$ for some integer $m$. The third property follows as in \cite[Corollary 2.24]{Tri}. The fourth property is \cite[Theorem 1]{HK}. For the fifth, see (\ref{beta1snl}), and for the sixth, see Proposition \ref{lemma: lower bound for eta}.

Milnor defined a signature function for knots $K$ (\cite[\S 5]{Miln}), closely related to $\sigma_K$, which extends to a function $\tau_L: S^1 \to \mathbb Z$ defined for links $L$ with non-zero Alexander polynomials such that 
$$|\tau_L(\zeta)| \leq \left\{ \begin{array}{ll} 
Z_\zeta(\Delta_L(t)) & \hbox{ if } \zeta = \pm 1 \\ 
2Z_\zeta(\Delta_L(t)) & \hbox{ if } \zeta \ne  \pm 1
\end{array} \right.$$ 
Here $Z_\zeta(\Delta_L(t))$ is the multiplicity of $\zeta$ as a zero of $\Delta_L(t)$, which is known to be $m-1$ when $\zeta = 1$. Further, if $0 \leq \varphi \leq \pi$ and $\zeta = \exp(\pi i \varphi)$ is not a root of $\Delta_L(t)$, Matumoto showed (\cite[Theorem 2]{Ma}) that 
\begin{equation} \label{matumoto}
\sigma_L(\exp(\pi i \varphi)) = \sum_{0 \leq \theta \leq \varphi} \tau_L(\exp(\pi i \theta))  
\end{equation}
Consequently, if $\Delta_L(t)$ is non-zero and $0 < \theta_0 < \pi$, for $\zeta, \zeta'$ contained in a small neighbourhood of $\exp(i \theta_0)$, $|\sigma_L(\zeta) - \sigma_L(\zeta')|$ is at most twice the multiplicity of $\exp(i \theta_0)$ as a root of $\Delta_K(t)$. Hence,

\begin{lemma} \label{constant}
Suppose that $L$ is a link of $m$ components contained in the $3$-sphere and $\zeta_0 \in S^1  \setminus \Delta_L^{-1}(0)$. Then $|\sigma_L(\zeta_0)| \leq \hbox{deg}(\Delta_L(t))$ with equality if and only if 
\vspace{-.2cm} 
\begin{itemize}

\item all the roots of $\Delta_L(t)$ are contained in $I_+(\zeta_0)$; 

\vspace{.2cm} \item $|\tau_L(1)| = m-1$; 

\vspace{.2cm} \item as $\zeta \ne \pm 1$ varies from $1$ to $-1$ through either hemisphere of $S^1$, the jumps in the values of $\sigma_L(\zeta)$ are all of the same sign and of absolute value equal to $2Z_\zeta(\Delta_L(t))$. If $m > 1$, the sign of these jumps is the same as the sign of $\tau_L(1)$. 

\end{itemize} 
\vspace{-.2cm}
Further, in the case that $|\sigma_L(\zeta_0)| = \hbox{deg}(\Delta_L(t))$, we have $|\sigma_L(\zeta)| = \hbox{deg}(\Delta_L(t))$ for all $\zeta \in \bar{I}_-(\zeta_0)$. 
\qed 
\end{lemma}

\section{The Murasugi-Tristram inequality} \label{sec: mt inequality}
Throughout this section $L$ will denote a link in the $3$-sphere with $m$ components. 

It follows from the definitions that if $F$ is a Seifert surface of $L$, then $|\sigma_L(\zeta)| + \eta_L(\zeta) \leq \beta_1(F)$ for all $\zeta$. If $\zeta$ is not a root of $\Delta_L(t)$ and $F$ is a locally flat, compact, oriented surface properly embedded in $B^4$ with oriented boundary $L$ and $\mu$ components, this can be sharpened to the {\it Murasugi-Tristram inequality},
\begin{equation} \label{mt}
|\sigma_L(\zeta)| + |\eta_L(\zeta) - (\mu -1)| \leq \beta_1(F) = 2g(F) + (m-\mu)
\end{equation} 
A version of (\ref{mt}) for $\zeta = -1$ was first proved by Murasugi (\cite[Theorem 9.1]{Mu2}) when $F$ is smooth. The case that $\zeta$ is a root of unity and $F$ is smooth can be found in \cite[Theorem 2.27]{Tri}. It is folklore that the inequality holds when $F$ is locally flat and though this case has been used in the literature (see, for example, \cite[Lemma 3.4]{Ta}), we are unaware of a proof having been published which is in full generality. Typically, it is assumed that $F$ is smooth. (See \cite[Theorem 5.19]{Fl} for instance.) A recent paper of Mark Powell provides a synopsis of the work that has been done on the Murasugi-Tristram inequality and proves it in the case that $F$ is locally flat with $m$ components (\cite[Theorem 1.4]{Po}). As the locally flat case of the (\ref{mt}) is key to our arguments, we include a self-contained proof below. The elements of the argument (\S \ref{sec: betti numbers} and \S \ref{subsec: mt inequality}) arise as in the smooth case. (Compare \cite[Corollary 1.4 and Theorem 3.4]{KT}, \cite[Proposition 1.5]{Gi}, and \cite[Theorem 5.19]{Fl}.)

For the remainder of this section, $F$ will be a locally flat, compact, oriented surface properly embedded in $B^4$ with oriented boundary $L$ and $\mu$ components. 

\subsection{Signatures and branched covers} \label{subsec: signatures and branched covers} 
Let $X$ be the exterior of $F$ in $B^4$ with orientation inherited from $\mathbb C^2$. Then 
$$\widetilde{H}_r(X) \cong H_{r+1}(B^4, X) \cong H_{r+1}(F \times (D^2, S^1)) \cong H_{r-1}(F) \cong \left\{ 
\begin{array}{ll} 
\mathbb Z^\mu & r = 1 \\ 
\mathbb Z^{\beta_1(F)} & r = 2 \\ 
0 & \hbox{otherwise} 
\end{array} \right.$$
where $H_1(X)$ is generated by the oriented meridians of the components of $F$. 

Let $\pi_1(X) \to \mathbb Z$ be the epimorphism which sends each of the oriented meridians of the components of $F$ to $1$ and $\widetilde X \to X$ the associated infinite cyclic cover. Reducing (mod $n$) yields an $n$-fold cyclic cover $X_n \to X$ and an associated $n$-fold cyclic cover $(\Sigma_n(F), F_n) \to (B^4, F)$ branched over $F$. Then $\partial \Sigma_n(F) = \Sigma_n(L)$. Both $X_n$ and $\Sigma_n(F)$ inherit orientations from $X$.

There is a Hermitian intersection form defined on $H_2(\Sigma_n(F); \mathbb C) \cong H_2(\Sigma_n(F)) \otimes_{\mathbb Z} \mathbb C$ given by 
$$\langle \xi \otimes z, \eta \otimes w \rangle_{\Sigma_n(F)} = z \bar w (\xi \cdot \eta)$$ 
where $\xi, \eta \in H_2(\Sigma_n(F))$, $z, w \in \mathbb C$, and $\xi \cdot \eta$ is the algebraic intersection of $\xi$ and $\eta$. This form is non-singular if and only if $\Sigma_n(L)$ is a rational homology $3$-sphere. That is, if and only if no $n^{th}$ root of unity is a root of $\Delta_L(t)$ (cf. \S \ref{sec: signature function}). Further, it is definite if and only if the intersection form on $H_2(\Sigma_n(F))$ is definite.

Let $\tau_n$ be the deck transformation of $\Sigma_n(F)$ which rotates a normal disk to $F_n$ by $\frac{2 \pi}{n}$. We use $t_n$ to denote any automorphism induced by $\tau_n$ on the homology of a $\tau_n$-invariant pair $(Y_n, Z_n)$ where $Y_n \subseteq \Sigma_n(F)$. If $R$ is a commutative ring, then letting $t$ act as $t_n$ makes $H_*(Y_n, Z_n; R)$ an $R[\mathbb Z] = R[t, t^{-1}]$-module. To simplify notation, we will refer to $t_n$ as $t$ below.  

For $0 \leq j \leq n-1$, define $E_r(Y_n, Z_n;j)$ to be the $\zeta_n^j$-eigenspace of $t: H_r(Y_n, Z_n; \mathbb C) \to H_r(Y_n, Z_n; \mathbb C)$ and 
$$\beta_r(Y_n, Z_n; j) = \hbox{dim}_{\mathbb C}(E_r(Y_n, Z_n;j)).$$ 
Clearly 
\begin{equation} \label{e0j}
\beta_0(\Sigma_n(F); j) = \left\{ 
\begin{array}{ll} 
1 & \hbox{ if } j = 0 \\
0 & \hbox{ if } 1 \leq j \leq n-1
\end{array} \right. \nonumber
\end{equation}
Since $H_*(\Sigma_n(F), X_n; \mathbb C) \cong H_*(F_n \times (D^2, S^1); \mathbb C)$, $t$ acts trivially on $H_*(\Sigma_n(F), X_n; \mathbb C)$. And since $\Sigma_n(F)$ is obtained from $X_n$ by attaching $2$-handles, the homomorphism $H_*(X_n; \mathbb C) \to H_*(\Sigma_n(F); \mathbb C)$ is surjective and induces isomorphisms
\begin{equation} \label{same eigenspaces}
E_r(X_n; j) \xrightarrow{\cong} E_r(\Sigma_n(F); j) \hbox{ {\rm for }} 1 \leq j \leq n-1 
\end{equation}

\subsubsection{The case that $F$ comes from a connected Seifert surface} \label{F connected} 
First consider the special case that $F$ is obtained by isotoping the interior of a {\it connected} Seifert surface of $L$ into the interior of $B^4$, $\pi_1(X) \cong \mathbb Z$ and is generated by any meridional class of $L$. In this case, $\Sigma_n(F)$ is simply connected for each $n \geq 2$. In particular, $0 = \beta_1(\Sigma_n(F)) = \beta_3(\Sigma_n(F), \Sigma_n(L))$. Thus we have an exact sequence
\begin{equation} \label{seifert surface case}
0 \to H_2(\Sigma_n(L); \mathbb C) \to H_2(\Sigma_n(F); \mathbb C) \xrightarrow{\psi} H_2(\Sigma_n(F), \Sigma_n(L); \mathbb C) \to H_1(\Sigma_n(L); \mathbb C) \to 0 
\end{equation}
The Hermitian intersection pairing $H_2(\Sigma_n(F); \mathbb C) \otimes H_2(\Sigma_n(F), \Sigma_n(L); \mathbb C) \to \mathbb C$ is non-singular, and so identifying $H_2(\Sigma_n(F), \Sigma_n(L); \mathbb C)$ with the dual space of $H_2(\Sigma_n(F); \mathbb C)$, $\psi$ is the adjoint of of the pairing $\langle \cdot, \cdot \rangle_{\Sigma_n(F)}$. On the other hand, Viro showed (\cite{Vi}; see also \cite{CS} and \cite{Ka}) that for $1 \leq j \leq n-1$,  $E_2(\Sigma_n(F); j) \cong H_1(F; \mathbb C)$ and the restriction of the Hermitian intersection pairing to $E_2(\Sigma_n(F); j)$ coincides with $\mathcal{S}_F(\zeta_n^j)$. Thus $\sigma_L(\zeta_n^j)$ is the signature of $\langle \cdot, \cdot\rangle_{\Sigma_n(F)}|E_2(\Sigma_n(F);j)$. 

\subsubsection{The general case} 
Suppose that $F$ is a locally flat surface with boundary $L$. In this case that $F$ is smooth, Viro applied the smooth $G$-signature theorem to deduce the following theorem. The extension to the locally flat case is obtained by replacing the smooth $G$-signature theorem with Wall's topological $G$-signature theorem \cite[Theorem 14B.2]{Wa}.  

\begin{thm} \label{thm: viro} {\rm (Viro)}
Suppose that $F$ is a locally flat, properly embedded, compact, oriented surface in $B^4$ with oriented boundary $L$. Then $\sigma_L(\zeta_n^j) = \hbox{{\rm signature}}\big(\langle \cdot, \cdot\rangle_{\Sigma_n(F)}|E_2(\Sigma_n(F));j)\big)$ for $1 \leq j \leq n-1$. 
\qed
\end{thm} 
Another consequence of \cite{Vi} and the sequence (\ref{seifert surface case}) is that when $F$ is obtained from a connected Seifert surface of $L$ we have $\eta_L(\zeta_n^j) = \hbox{{\rm nullity}}\big(\langle \cdot, \cdot\rangle_{\Sigma_n(F)}|E_2(\Sigma_n(F);j)\big) = \beta_2(\Sigma_n(L); j) = \beta_1(\Sigma_n(L); j)$ for $1 \leq j \leq n-1$. Hence 
\begin{equation} \label{eta}
\eta_L(\zeta_n^j) = \beta_1(\Sigma_n(L); j) \;\; \hbox{ for } 1 \leq j \leq n-1
\end{equation}
A transfer argument shows that $E_1(M_n; 0) \cong \mathbb C^{m}$ is generated by the inverse images of the meridians of the components of $L$. Hence $E_1(\Sigma_n(L); 0) \cong 0$ and therefore 
\begin{equation} \label{beta1snl}
\beta_1(\Sigma_n(L)) = \sum_{j=1}^{n-1} \eta_L(\zeta_n^j) 
\end{equation}

\subsection{Milnor exact sequences} \label{sec: Milnor sequence} 

Given an infinite cyclic cover of pairs $(\widetilde Y, \widetilde Z) \xrightarrow{f}  (Y, Z)$ and commutative ring $R$, Milnor defined an exact sequence 
$$ \ldots \to H_{r+1}(Y, Z; R) \to H_{r}(\widetilde Y, \widetilde Z; R) \xrightarrow{1-t} H_{r}(\widetilde Y, \widetilde Z; R) \xrightarrow{f_*} H_{r}(Y, Z; R) \to  \ldots $$
of $R[t, t^{-1}]$-modules where $t$ is defined to act as the identity on $H_{*}(Y, Z; R)$ (\cite[\S 2]{Miln}). 

\begin{lemma} 
\label{lemma: wedge of circles} 
Let $Y$ be a wedge of $\mu$ circles and $\widetilde Y \to Y$ an infinite cyclic cover. For each $n \geq 2$ let  $Y_n \to Y$ be the associated $n$-fold cyclic cover. Then 
$$\beta_1(Y_n; j) = \left\{ \begin{array}{cl} \mu & \mbox{ if } j = 0 \\ \mu - 1 & \mbox{ if } 1 \leq j \leq n-1 \end{array} \right.$$
\end{lemma}

\begin{proof} 
Let $\Omega$ be the quotient field of the principal ideal domain $\mathbb C[t, t^{-1}]$ and set $H_*(Y;  \Omega) = H_*(\widetilde Y; \mathbb C) \otimes_{\mathbb C[t,t^{-1}]} \Omega$. Since $H_0(Y; \Omega) = 0$, an Euler characteristic argument shows that $H_1(Y; \Omega)  \cong \Omega^{\mu -1}$. Thus $H_1(\widetilde Y; \mathbb C) \cong \mathbb C[t,t^{-1}]^{\mu -1} \oplus T$ where $T$ is a torsion $\mathbb C[t,t^{-1}]$-module. 

Another appeal to an Euler characteristic argument shows that $\beta_1(Y_n) = n(\mu -1) + 1$, so from the Milnor exact sequence associated to $\widetilde Y \to Y_n$ we obtain an exact sequence of $\mathbb C[t,t^{-1}]$-modules
$$H_1(\widetilde Y; \mathbb C) \xrightarrow{1-t^n} H_1(\widetilde Y; \mathbb C) \to H_1(Y_n; \mathbb C) \cong \mathbb C^{n(\mu -1) + 1} \to H_0(\widetilde Y; \mathbb C) \cong \mathbb C \to 0$$ 
In particular, the cokernel of $H_1(\widetilde Y; \mathbb C) \xrightarrow{1-t^n} H_1(\widetilde Y; \mathbb C)$ is $\mathbb C^{n(\mu -1)}$. From above, this cokerrnel is isomorphic to $\mathbb C[t,t^{-1}]^{\mu -1}/(1-t^n)( \mathbb C[t,t^{-1}]^{\mu -1}) \oplus T/(1-t^n)T \cong \mathbb C^{n(\mu -1)} \oplus T/(1-t^n)T$. Thus $T/(1-t^n)T = 0$. Milnor's sequence then shows that as a $\mathbb C[t,t^{-1}]$-module we have 
$$H_1(Y_n; \mathbb C) \cong H_0(\widetilde Y; \mathbb C) \oplus  \big(\mathbb C[t,t^{-1}]/(1-t^n)\big)^{\mu - 1}  \cong \mathbb C[t,t^{-1}]/(t - 1) \oplus \big( \oplus_{j = 0}^{n-1} \big( \mathbb C[t,t^{-1}]/(t - \zeta_n^j)\big)^{\mu - 1} \big)$$
Therefore the lemma holds. 
\end{proof} 

\begin{lemma} \label{lemma: lower bound for eta} 
Suppose that $L$ has a Seifert surface with $\mu$ components. Then $\eta(\zeta_n^j) \geq  \mu -1$ for each $n \geq 2$ and $1 \leq j \leq n-1$. Consequently $\beta_1(\Sigma_n(L)) \geq (n-1)(\mu-1)$ for each $n \geq 2$. Further, if $\zeta$ is not a root of $\Delta_L(t)$, then $\eta_L(\zeta) \geq \mu -1$.  
\end{lemma}

\begin{proof} 
Since roots of unity are dense in the circle and $\eta_L$ is constant on $S^1 \setminus \Delta_L^{-1}(0)$, the last assertion of the lemma follows from the first. 

Let $M$ be the exterior of $L$ and denote by $i: Y \to M$ the inclusion of a wedge of $\mu$ circles obtained by joining the ends of a family of transverse arcs, one for each component of $F$, to a common base point using arcs in $M \setminus F$. There is a map $r: M \to Y$ which sends the complement of a tubular neighbourhood of $F$ to the wedge point of $Y$ and for which the composition $r \circ i$ is homotopic to the identity of $Y$. 

Fix $n \geq 2$ and recall the $n$-fold cyclic cover  $M_n \to M$ constructed in the first paragraph of \S \ref{sec: signature function}. The inclusion $i$ induces a connected $n$-fold cyclic cover $Y_n \to Y$ such that $i$ lifts to an inclusion $i_n: Y_n \to M_n$. It is easy to see that $r$ lifts to a map $r_n: M_n \to Y_n$ such that $r_n \circ i_n: Y_n \to Y_n$ is homotopic to the identity. Hence $(i_n)_*:H_1(Y_n; \mathbb C) \to H_1(M_n; \mathbb C)$ is injective and so for each $j$ there is an injection $E_1(Y_n; j) \to E_1(M_n; j)$. Further, for $1 \leq j \leq n-1$ we have $E_1(\Sigma_n(L); j) \cong E_1(M_n; j)$ and so for such $j$, $\beta_1(\Sigma_n(L); j) \geq \beta_1(Y_n; j) = \mu -1$ (Lemma \ref{lemma: wedge of circles}). Then by (\ref{eta}), $\eta_L(\zeta_n^j) \geq  \mu -1$ for $1 \leq j \leq n-1$. It then follows from (\ref{beta1snl}) that  $\beta_1(\Sigma_n(L)) \geq (n-1)(\mu-1)$
\end{proof} 

\begin{cor} \label{cor: conn ss}
If $\Sigma_n(L)$ is a rational homology $3$-sphere for some $n \geq 2$, then each Seifert surface for $L$ is connected. 
\qed 
\end{cor} 

\subsection{The betti numbers of $\Sigma_n(F)$} \label{sec: betti numbers}

\begin{lemma}  \label{lemma: finite sum of cyclics} 
For each $r$, $\beta_r(\Sigma_n(F); j) = \beta_r(\Sigma_n(F); j')$ if $\gcd(j,n) = \gcd(j',n)$. 
\end{lemma}

\begin{proof}
For $r \in \mathbb Z$, $H_r(\Sigma_n(F); \mathbb Q)$ is a finitely generated module over the principal ideal domain $\mathbb Q[t, t^{-1}]$ in the obvious way. Since it is annihilated by $t^n - 1$, it has no $\mathbb Q[t, t^{-1}]$ summands and thus there is a $\mathbb Q[t, t^{-1}]$-module isomorphism $H_r(\Sigma_n(F); \mathbb Q) \cong \bigoplus_{k} \mathbb Q[t, t^{-1}]/(p_k(t)^{m_k})$ where $p_k(t)$ is  an irreducible element of $\mathbb Q[t, t^{-1}]$ and $m_k \geq 1$. (If $H_r(\Sigma_n(F); \mathbb Q) = 0$, this is an empty sum.) Again, since $t^n - 1$ annihilates $H_r(\Sigma_n(F); \mathbb Q)$, $p_k(t)^{m_k}$ divides $t^n - 1$ for each $k$. It follows that each $m_k = 1$ and $p_k(t)$ is the cyclotomic polynomial $\Phi_d(t)$ for some $d$ dividing $n$. Thus $H_r(\Sigma_n(F); \mathbb Q)\cong \bigoplus_{k} \mathbb Q[t, t^{-1}]/(\Phi_{d_k}(t))$ where each $d_k$ divides $n$. 

As a $\mathbb C[t, t^{-1}]$-module, $\mathbb Q[t, t^{-1}]/(\Phi_{d_k}(t)) \otimes_{\mathbb Q} \mathbb C \cong \mathbb C[t, t^{-1}]/(\Phi_{d_k}(t))\cong \bigoplus_l \mathbb C[t, t^{-1}]/(t - \zeta_{n}^l)$ where $\zeta_{n}^l$ ranges over the primitive $d_k^{th}$ roots of unity. On the other hand, each $E_r(\Sigma_n(F);j)$ is a $\mathbb C[t, t^{-1}]$-module isomorphic to a sum of copies of $\mathbb C[t, t^{-1}]/(t - \zeta_{n}^j)$. Hence $\beta_r(\Sigma_n(F); j)$ equals the number of $k$ for which $d_k = \gcd(j, n)$. It follows that if $1 \leq j, j' \leq n-1$, then $\beta_r(\Sigma_n(F); j) = \beta_r(\Sigma_n(F); j') \hbox{ {\rm if }} \gcd(j,n) = \gcd(j',n)$. 
\end{proof}

\begin{lemma}  \label{lemma: chi} 
$\beta_2(\Sigma_n(F); j) = \beta_1(F) - (\mu - 1) + \beta_1(\Sigma_n(F); j) + \beta_3(\Sigma_n(F); j)$ for $1 \leq j \leq n-1$.
\end{lemma} 

\begin{proof}
First note that since $\chi(\Sigma_n(F)) = \chi(X_n) + \chi(F \times D^2) - \chi(F \times S^1) = n\chi(X) + \chi(F)$ where $\chi(X) = 1 - \chi(F) = \beta_1(F) - \mu + 1$, we have 
$\chi(\Sigma_n(F)) = 1 + (n-1)\chi(X) =1 + (n-1)(\beta_1(F) - (\mu - 1))$. Thus 
\begin{equation} \label{beta2j} 
\beta_2(\Sigma_n(F)) = (n-1)(\beta_1(F) - (\mu - 1)) + \beta_1(\Sigma_n(F)) + \beta_3(\Sigma_n(F))
\end{equation}
A transfer argument shows that 
\begin{equation} \label{betar0}
\beta_r(\Sigma_n(F);0) =  \left\{ 
\begin{array}{ll} 
1 & \hbox{ if } r = 0 \\
0 & \hbox{ if } r > 0
\end{array} \right.
\end{equation}
We proceed by induction on $n$. 

When $n = 2$, (\ref{betar0}) implies that $\beta_2(\Sigma_2(F);1) = \beta_2(\Sigma_2(F)) = \beta_1(F) - (\mu - 1) + \beta_1(\Sigma_2(F)) + \beta_3(\Sigma_2(F)) = \beta_1(F) - (\mu - 1) + \beta_1(\Sigma_2(F); 1) + \beta_3(\Sigma_2(F); 1)$, so we are done. 

Fix $n > 2$ and suppose that the lemma holds for all $n' < n$. If $d \geq 2$ is a divisor of $n$, the transfer map associated to the cover $\Sigma_n(F) \to \Sigma_{\frac{n}{d}}(F)$ induces isomorphisms $E_r(\Sigma_{\frac{n}{d}}(F);k) = E_r(\Sigma_n(F);dk)$ for $1 \leq k \leq \frac{n}{d}$ and each $r$. Hence if $\gcd(j, n) = d > 1$, (\ref{beta2j}) and our inductive hypothesis implies that  
\begin{eqnarray} 
\beta_2(\Sigma_n(F);j) = \beta_2(\Sigma_{\frac{n}{d}}(F); {\frac{j}{d}}) & = & \beta_1(F) - (\mu - 1) + \beta_1(\Sigma_{\frac{n}{d}}(F); {\frac{j}{d}}) + \beta_3(\Sigma_{\frac{n}{d}}(F); {\frac{j}{d}}) \nonumber \\ 
& = & \beta_1(F) - (\mu - 1) + \beta_1(\Sigma_n(F); j) + \beta_3(\Sigma_n(F); j) \nonumber
\end{eqnarray}
as claimed. On the other hand, if $\gcd(j, n) = 1$ and $\varphi$ is Euler's totient function, we have 
{\small 
\begin{eqnarray} 
\beta_2(\Sigma_n(F))& = & \sum_{i = 1}^{n-1} \beta_2(\Sigma_n(F); i)  \;\; \hbox{ {\rm by (\ref{beta2j})}} \nonumber \\
& = & \varphi(n) \beta_2(\Sigma_n(F); j) + \sum_{\gcd(i,n) > 1} \beta_2(\Sigma_n(F); i) \;\; \hbox{ {\rm by Lemma \ref{lemma: finite sum of cyclics}}} \nonumber \\
& = & \varphi(n) \beta_2(\Sigma_n(F); j) + \sum_{\gcd(i,n) > 1} \beta_2(\Sigma_{\frac{n}{\gcd(i,n)}}(F); \frac{i}{\gcd(i,n)}) \nonumber \\
& = & \varphi(n) \beta_2(\Sigma_n(F); j) + \sum_{\gcd(i,n) > 1} \big( \beta_1(F) - (\mu - 1) + \beta_1(\Sigma_n(F); i) + \beta_3(\Sigma_n(F); i) \big) \nonumber  
\end{eqnarray}
}
where the last equality follows by induction. Now 
$$\sum_{\gcd(i,n) > 1} \beta_1(\Sigma_n(F); i) = \beta_1(\Sigma_n(F)) - \varphi(n)\beta_1(\Sigma_n(F); j)$$ 
and 
$$\sum_{\gcd(i,n) > 1} \beta_3(\Sigma_n(F); i) =\beta_3(\Sigma_n(F)) - \varphi(n)\beta_3(\Sigma_n(F); j),$$
so combining the last three identities with (\ref{beta2j}) and dividing by $\varphi(n)$ we obtain
the desired result. 
\end{proof}

\begin{remark} \label{rem: seifert surface case}
{\rm Suppose that $F$ is obtained by isotoping the interior of a connected Seifert surface of $L$ into the interior of $B^4$. From \S \ref{subsec: signatures and branched covers} we know that $\beta_1(\Sigma_n(F)) = 0$ and so as $\Sigma_n(L)$ is connected, $\beta_3(\Sigma_n(F)) =  \beta_1(\Sigma_n(F), \Sigma_n(L)) = 0$ as well. Hence by Lemma \ref{lemma: chi}, 
$$\beta_2(\Sigma_n(F); j) = \beta_1(F) = 2g(F) + (m-1)$$ 
for $1 \leq j \leq n-1$.}
\end{remark} 

\begin{lemma}  \label{lemma: prime power covers} 
Suppose that $n$ is a prime power. Then 

$(1)$ $\beta_3(\Sigma_n(F); j) = 0$ for all $j$.

$(2)$ $\beta_1(\Sigma_n(F); j) \leq \min\{\eta_L(\zeta_n^j), \mu -1\}$ for $1 \leq j \leq n-1$. 

\end{lemma} 

\begin{proof}
Let $n= p^m$ where $p$ is prime and set $\mathbb F_p = \mathbb Z/p \mathbb Z$. A transfer argument implies that $E_3(\Sigma_n(F);0) \cong H_3(B^4; \mathbb C) = 0$ and as $E_3(\Sigma_n(F); j) \cong E_3(X_n; j)$ for $1 \leq j \leq n-1$ by (\ref{same eigenspaces}), to prove (1) it suffices to show that $H_3(X_n; \mathbb F_p) = 0$. 

Since $H_3(X; \mathbb F_p) = 0$, it follows from the Milnor sequence associated to the cover $\widetilde X \to X$ that $H_3(\widetilde X; \mathbb F_p) \xrightarrow{1-t} H_3(\widetilde X; \mathbb F_p) $ is surjective and $H_2(\widetilde X; \mathbb F_p)  \xrightarrow{1-t} H_2(\widetilde X; \mathbb F_p) $ is injective. But then as $1-t^n \equiv (1-t)^n$ (mod $p$), the Milnor sequence associated to $\widetilde X \to X_n$ implies that $H_3(X_n; \mathbb F_p) = 0$, which yields (1). 

Next we prove (2). By duality, $E_1(\Sigma_n(F), \Sigma_n(L); j) \cong E_3(\Sigma_n(F); j) = 0$, so the homomorphism $E_1(\Sigma_n(L); j) \to E_1(\Sigma_n(F); j)$ is surjective. Thus $\beta_1(\Sigma_n(F); j) \leq \beta_1(\Sigma_n(L); j) = \eta_L(\zeta_n^j)$ (cf. (\ref{eta})). To complete the proof of (2), we show that $\beta_1(\Sigma_n(F); j) \leq \mu -1$. Since $\Sigma_n(F)$ is obtained from $X_n$ by attaching $2$-handles, it suffices to show that $\beta_1(X_n; j) \leq \mu -1$. 

Let $Y$ be a wedge of $\mu$ circles contained in $X$ such that the inclusion $Y \to X$ induces an isomorphism in integer homology. Let $\widetilde Y \to Y$ be the infinite cyclic cover induced by $\widetilde X \to X$ and $Y_n \to Y$ the $n$-fold cyclic cover induced by $X_n \to X$. The inclusion $Y \to X$ lifts to inclusions $\widetilde Y \to \widetilde X$ and $Y_n \to X_n$. Since $H_1(X, Y; \mathbb F_p) = 0$, the Milnor exact sequence associated to $(\widetilde X, \widetilde Y) \to (X, Y)$ shows that the homomorphism $H_1(\widetilde X, \widetilde Y; \mathbb F_p) \xrightarrow{1-t} H_1(\widetilde X, \widetilde Y; \mathbb F_p)$ is surjective. Hence as $1-t^n \equiv (1-t)^n$ (mod $p$), the Milnor sequence of $(\widetilde X, \widetilde Y) \to (X_n, Y_n)$ shows that $H_1(X_n, Y_n; \mathbb F_p) \cong 0$. Then $H_1(X_n, Y_n; \mathbb C) \cong 0$, so the natural map $H_1(Y_n; \mathbb C) \to H_1(X_n; \mathbb C)$ is onto. Then $\beta_1(X_n; j) \leq \beta_1(Y_n; j)$ and by Lemma \ref{lemma: wedge of circles}, $\beta_1(Y_n; j) = \mu -1$, so we are done. 
\end{proof}

\subsection{The Murasugi-Tristram Inequality} \label{subsec: mt inequality}

\begin{prop}  {\rm (The Murasugi-Tristram Inequality)} \label{prop: mt inequality} 
Suppose that $F$ is a locally flat, compact, oriented surface properly embedded in $B^4$ with oriented boundary $L$ and $\mu$ components. If $\zeta$ is not a root of $\Delta_L(t)$, then  
$$|\sigma_L(\zeta)| + |\eta_L(\zeta) - (\mu -1)| \leq \beta_1(F) = 2g(F) + (m-\mu)$$
\end{prop} 

\begin{proof}
Since $\sigma_L$ and $\eta_L$ are constant on the components of $S^1 \setminus \Delta_L^{-1}(0)$ and prime power roots of unity are dense in the circle, it suffices to show that the inequality holds for $\zeta = \zeta_n^j$ where $n= p^m$ with $p$ prime and $1 \leq j \leq n-1$. In this case, $H_1(\Sigma_n(F), \Sigma_n(L); j) \cong H_3(\Sigma_n(F); j) = 0$ (Lemma \ref{lemma: prime power covers} (1)). Hence we have an exact sequence
$$E_2(\Sigma_n(F); j) \xrightarrow{\psi} E_2(\Sigma_n(F), \Sigma_n(L); j) \to E_1(\Sigma_n(L); j) \to E_1(\Sigma_n(F); j) \to 0$$
We noted in \S \ref{subsec: signatures and branched covers} that $\psi$ can be identified with the adjoint homomorphism of the pairing $\langle \cdot , \cdot \rangle_{\Sigma_n(F)}$ restricted to $E_2(\Sigma_n(F); j)$. Thus $|\sigma_L(\zeta_n^j)|$ is bounded above by the rank of $\psi$. Hence
\begin{eqnarray} 
|\sigma_L(\zeta_n^j)| & \leq & \beta_2(\Sigma_n(F), \Sigma_n(L); j) - \beta_1(\Sigma_n(L); j) + \beta_1(\Sigma_n(F); j) \nonumber \\
& = & \beta_2(\Sigma_n(F); j) - \beta_1(\Sigma_n(L); j) + \beta_1(\Sigma_n(F); j) \nonumber \\
& = &  \beta_1(F) - (\mu - 1) - \eta_L(\zeta_n^j) + 2\beta_1(\Sigma_n(F); j)  \hbox{ {\rm by (\ref{eta}) and Lemmas \ref{lemma: prime power covers}(1) and  \ref{lemma: chi}}} \nonumber 
\end{eqnarray}
Since $\beta_1(\Sigma_n(F); j) \leq \eta_L(\zeta_n^j)$ (Lemma \ref{lemma: prime power covers}(2)), we obtain
$$|\sigma_L(\zeta_n^j)| -(\eta_L(\zeta_n^j) - (\mu - 1)) \leq \beta_1(F)$$ 
And since $\beta_1(\Sigma_n(F); j) \leq \mu-1$ (Lemma \ref{lemma: prime power covers}(2)), we also obtain
$$|\sigma_L(\zeta_n^j)| +(\eta_L(\zeta_n^j) - (\mu - 1)) \leq \beta_1(F)$$ 
The last two inequalities imply the conclusion of the proposition. 
\end{proof} 

\begin{prop} \label{g4 bound}  
If $K$ is a knot, then $|\sigma_K(\zeta)| \leq 2g_4^{top}(K)$ for all $\zeta \in S^1$ which are not roots of $\Delta_K(t)$. 
\end{prop}

\begin{proof}
Since $\eta_K(\zeta) = 0$ when $\zeta$ is not a root of $\Delta_K(t)$, Proposition \ref{prop: mt inequality} shows that $|\sigma_K(\zeta)| \leq 2g(F)$ for each $\zeta \in S^1 \setminus \Delta_K^{-1}(0)$ and any locally flat surface $F$ properly embedded in $B^4$ with boundary $L$. 
Hence $|\sigma_K(\zeta)| \leq 2g_4^{top}(K)$ for such $\zeta$. 
\end{proof}
\vspace{-.2cm} 
Recall that when $\Sigma_n(L)$ is a rational homology $3$-sphere, $\eta_L(\zeta_n^j) = 0$ for $1 \leq j \leq n-1$. 

\begin{cor} \label{cor: florens inequality qhs} 
Suppose that $\Sigma_n(L)$ is a rational homology $3$-sphere and that $F$ is a locally flat, compact, oriented surface properly embedded in $B^4$ with oriented boundary $L$ and $\mu$ components.  Then for $1 \leq j \leq n-1$, 
$$|\sigma_L(\zeta_n^j)| + 2(\mu -1) \leq 2g(F) + (m-1)$$ 
\qed
\end{cor}

\subsection{Links with maximal signatures} \label{sec: maximal signatures} 

Recall the definitions of $I_\pm(\zeta)$ for $\zeta \in S^1$ from the introduction. 

\begin{prop} \label{proposition: maximal signature}
Let $g(L)$ be the genus of a link of $m$ components $L$ and suppose that $|\sigma_L(\exp(i \theta_0))| = 2g(L) + (m-1)$ for some $\theta_0$. Then, 

$(1)$ $\exp(i \theta_0)$ is not a root of $\Delta_L(t)$ and therefore $\eta(\exp(i \theta_0)) = 0$; 

$(2)$ $g_4^{top}(L) = g(L)$. Further, any locally flat, compact, oriented surface $F$ properly embedded in $B^4$ with oriented boundary $L$ which realises $g_4^{top}(L)$ is connected;  

$(3)$ $|\tau_L(1)| = m-1$ and as $\zeta \ne \pm 1$ varies from $1$ to $-1$ through either hemisphere of $S^1$, the jumps in the values of $\sigma_L(\zeta)$ are all of the same sign. Further, the absolute value of the jump at $\zeta \ne \pm 1$ is $2Z_\zeta(\Delta_L(t))$ and if $m > 1$, its sign is the same as that of $\tau_L(1)$. 

$(4)$ all the roots of $\Delta_L(t)$ lie in $I_+(\exp(i \theta_0))$. Further, $|\sigma_L(\zeta)| = 2g(L) + (m-1) = \hbox{{\rm deg}}( \Delta_L(t)) = 2g(L) + (m-1)$ and $\eta(\zeta) = 0$ for all $\zeta \in \bar{I}_-(\exp(i \theta_0))$;  

$(5)$ if $\Delta_L(t)$ is monic, it is a non-trivial product of cyclotomic polynomials. 
\end{prop}

\begin{proof}
First observe that $\exp(i \theta_0)$ cannot be a root of $\Delta_L(t)$. Otherwise, if $A$ is a Seifert matrix for $L$ of size $(2g(L) + (m-1)) \times (2g(L) + (m-1))$, then the identity $2g(L) + (m-1) = |\hbox{signature}((1 - \exp(i \theta_0))A + (1 - \exp(-i \theta_0))A^T)|$ implies that $(1 - \exp(i \theta_0))A + (1 - \exp(-i \theta_0))A^T$ is definite. But then for some integer $k$, 
$$0 \ne \det((1 - \exp(i \theta_0))A + (1 - \exp(-i \theta_0))A^T) = \exp(i \theta_0)^{k}(1 - \exp(i \theta_0))^{(2g(L) + (k-1))}\Delta_L(\exp(i \theta_0))$$ 
This proves (1). 

Since $\exp(i \theta_0)$ is not a root of $\Delta_L(t)$, it follows from Proposition \ref{prop: mt inequality} that if $F$ is a locally flat, compact, oriented surface properly embedded in $B^4$ with oriented boundary $L$ and $\mu$ components, then by Proposition \ref{prop: mt inequality}, 
\begin{eqnarray} 
2g(L) + (m-1) = |\sigma_L(\exp(i \theta_0)|  &\leq & |\sigma_L(\exp(i \theta_0)| + (\mu -1) \nonumber \\
& = & |\sigma_L(\exp(i \theta_0)| + |\eta_L(\exp(i \theta_0)) - (\mu -1)| \nonumber \\ 
&\leq&  2g(F) + (m-\mu) \nonumber \\ 
&\leq& 2g(F) + (m-1). \nonumber 
\end{eqnarray} 
Thus $g(L) \leq g(F)$ and so taking $F$ which realises $g_4^{top}(L)$, we see that $g(L) \leq g_4^{top}(L)$ and $\mu(F) = 1$. Hence $g(L) = g_4^{top}(L)$ and $F$ is connected. Further, since $|\sigma_L(\exp(i \theta_0)| \leq \hbox{{\rm deg}} (\Delta_L(t)) \leq 2g(L) + (m-1) = |\sigma_L(\exp(i \theta_0)|$, we have $\hbox{{\rm deg}} (\Delta_L(t)) = 2g(L) + (m-1) = |\sigma_L(\exp(i \theta_0)|$. Thus (2) holds. 

Assertions (3) and (4) follow from Lemma \ref{constant}. 

Finally note that if $\Delta_L(t)$ is monic, Kronecker's theorem (\cite[page 118]{Pd}) combines with (2) and (4) to show that $\Delta_L(t)$ is a non-trivial product of cyclotomic polynomials, which is (5). 
\end{proof} 

\section{Genera of links} \label{sec: genera}

For a compact orientable surface $F$ with no closed components, the {\it big genus of $F$} (\cite[\S 5A]{BW}), denoted $G(F)$, is the genus of the closed surface obtained by attaching a connected planar surface with $m = |\partial F|$ boundary components to $F$. If $F$ has $\mu$ components, then $G(F) = g(F) + (m - \mu) = \frac12(m - \chi(F))$. 

The {\it big genus} of $L$, denoted $G(L)$, is the minimum value of $G(F)$ where $F$ is a smooth oriented surface with no closed components contained in the $3$-sphere with oriented boundary $L$.  

We use $G_4(L)$ to denote the minimum value of $G(F)$ where $F$ is a smooth oriented surface with no closed components properly embedded in the $4$-ball with oriented boundary $L$. If $\chi_4(L)$ denotes the greatest value of $\chi(F)$ among such $F$, then $G_4(L) = \frac12(m - \chi_4(L))$.

Similarly we use $G_4^{top}(L)$ to denote the minimum value of $G(F)$ where $F$ is a locally flat oriented surface with no closed components properly embedded in the $4$-ball with oriented boundary $L$. 

\begin{prop} \label{prop: florens inequality qhs 2} 
Suppose that $\Sigma_n(L)$ is a rational homology $3$-sphere. Then for $1 \leq j \leq n-1$ we have 
$$|\sigma_L(\zeta_n^j)| + (m-1) \leq 2G_4^{top}(L)$$
\end{prop}

\begin{proof}
Suppose that $F$ is a locally flat surface with $\mu$ components, none closed, which is properly embedded in the $4$-ball and has oriented boundary $L$. Suppose as well that $G(F) = G_4^{top}(L)$. 
Then by Corollary \ref{cor: florens inequality qhs} we have $|\sigma_L(\zeta_n^j)| + 2(\mu -1) \leq 2g(F) + (m-1) = 2G(F) + 2(\mu - 1) - (m-1)$, which implies the desired inequality. 
\end{proof}

\begin{prop} {\rm (Kronheimer-Mrowka \cite[Corollary 1.3]{KM})}
\label{prop: rudolph}
Let $L$ be a link which is the oriented boundary of the intersection $F$ of a complex affine curve with a smooth $4$-ball in $\mathbb C^2$. Then $\chi_4(L) = \chi(F)$. Consequently,
$G_4(L) = G(F)$, and if $F$ is isotopic rel $\partial F$ to a Seifert surface for $L$, then $G(L) = G_4(L) = G(F)$. 
\end{prop}

\begin{proof}
The fact that $\chi_4(L) = \chi(F)$ is \cite[Proposition, \S 3]{Ru3}. Since $G_4(L) = \frac12(m - \chi_4(L))$, this implies $G_4(L) = G(F)$. If $F$ is isotopic rel $\partial F$ to a Seifert surface for $L$, then $G(L) \leq G(F) = G_4(F) \leq G(L)$. This completes the proof. 
\end{proof}

\section{Notions of positivity} \label{sec: positivity} 

A link is called a {\it positive braid link} if it is the closure of a braid which can be expressed as a product of positive powers of the standard generators $\sigma_i$ of the braid group and all strings of the braid are like-oriented.  

A link is called {\it positive} if it has a diagram all of whose crossings are positive.

A braid is called {\it strongly quasipositive} if it is the product of conjugates of positive powers of the standard generators $\sigma_i$ of the braid group, where each conjugating element is of the form $\sigma_j \sigma_{j+1} \cdots \sigma_{i-1}$. A link $L$ is called {\it strongly quasipositive} if it is the closure of a strongly quasipositive braid where the braid components are like-oriented. Equivalently, a link $L$ in the $3$-sphere is strongly quasipositive if it bounds a Seifert surface obtained from a finite number of parallel like-oriented disks by attaching positive half-twisted bands. According to (\cite{Ru1}, \cite{BO}), when $S^3$ is viewed as the strictly pseudoconvex boundary of $B^4 \subset \mathbb R^4 = \mathbb C^2$, this Seifert surface can be isotoped into  $B^4$, relative to $L$, to a properly embedded surface $F$ which is the intersection of $B^4$ with a complex affine curve in $\mathbb C^2$. In this case, 
$$G(F) = G_4(L) = G(L)$$ 
by Proposition \ref{prop: rudolph}. 

It follows from \cite[Theorem 1.2]{He}, \cite[Corollary 1.3]{Ni} and the calculations of \cite{OS3} that L-space knots are strongly quasipositive. 

A braid is called {\it quasipostive} if it is the product of conjugates of the standard generators of the braid group. A link is called {\it quasipositive} if it is the closure of a quasipositive braid. Quasipositive links are precisely the class of links which bound the intersection $F$ of a smooth complex affine curve in $\mathbb C^2$ with $B^4$ (\cite{Ru1}, \cite{BO}). In this case, 
$$G(F) = G_4(L)$$ 
by Proposition \ref{prop: rudolph}. 

It is evident that positive braid links are positive links and that strongly quasipositive links are quasipositive links. What is less obvious is that positive links are strongly quasipositive links, but this has been shown by Nakamura \cite{Nak} and Rudolph \cite{Ru4}. 

A family of fibred strongly quasipositive knots which will arise below consists of A'Campo's knots of divides, or {\it divide knots} (\cite{A1}, \cite{A2}). These knots are fibred \cite{A1} and Kawamura proved that they are quasipositive and $g_4(K) = g(K)$ (\cite{Kaw}). Plamenevskaya \cite{Pl} has shown that $\tau(K) = g_4(K)$ for quasipositive knots, so by Hedden \cite[Theorem 1.2]{He}, divide knots are strongly quasipositive.  

\section{Strongly quasipositive links with L-space branched covers} \label{sec: sqp}

Throughout this section $L$ will be a strongly quasipositive link. Up to replacing $L$ by its mirror image, we can fix a Seifert surface of $L$ which can be isotoped, relative to $L$, to be a properly embedded surface $F \subset B^4$ which equals the intersection of a complex affine curve in $\mathbb{C}^2$ with $B^4$. Then $G_4(L) = G(F) = G(L)$ (cf. \S \ref{sec: positivity}). Also, $\Sigma_n(F)$ is a Stein filling of $\Sigma_n(L)$ (\cite[Theorem 1.3]{HKP}, \cite[Theorem 1.3]{Ru7}) which is simply-connected (\S \ref{F connected}).

\begin{prop} \label{prop: main prop}
Suppose that $L$ is a strongly quasipositive link. If $\Sigma_n(L)$ is an L-space for some $n \geq 2$, then $|\sigma_L(\zeta_n^j)| = 2 g(L) + (m-1)$ for $1 \leq j \leq n-1$. Consequently, $L$ satisfies the conclusions of Proposition \ref{proposition: maximal signature} with $\exp(i \theta_0) = \zeta_n$. 
\end{prop}

\begin{proof}
It follows from the fact that $\Sigma_n(L)$ is a rational homology $3$-sphere that 
\vspace{-.2cm}
\begin{itemize}

\item $L$ bounds only connected surfaces in the $3$-sphere, so $F$ is connected (Corollary \ref{cor: conn ss});

\vspace{.2cm} \item $\eta_L(\zeta_n^j) = 0$ for $1 \leq j \leq n-1$ (\S \ref{sec: signature function});

\vspace{.2cm} \item $\Delta_L(\zeta_n^j) \ne 0$ for $1 \leq j \leq n-1$ (\S \ref{sec: signature function}).

\end{itemize}
\vspace{-.2cm}
Then as $G(L) = G(F)$, we have $g(L) + (m-1) = G(L) = G(F) = g(F) + (m-1)$, so $g(L) = g(F)$. 

The condition that $\Sigma_n(L)$ be a rational homology $3$-sphere also implies that the intersection form on $H_2(\Sigma_n(F))$ is non-singular, so as $\Sigma_n(F)$ is Stein, $\beta_2^+(\Sigma_n(F)) = 0$ by \cite[Theorem 1.4]{OS1}. (This means that the intersection form on the second homology of $\Sigma_n(F)$ is negative definite.) Hence for $1 \leq j \leq n-1$, 
$$|\sigma_L(\zeta_n^j)| = \beta_2(\Sigma_n(F); j) = 2g(F) + (m-1) = 2g(L) + (m-1)$$
by Remark \ref{rem: seifert surface case}. In particular, the hypotheses of Proposition \ref{proposition: maximal signature} hold for $\exp(i \theta_0) = \zeta_n$, which completes the proof. 
\end{proof}

\begin{proof}[Proof of Theorem \ref{thm: sqp links}]
Suppose that $\Sigma_n(L)$ is an L-space for some $n \geq 2$. Proposition \ref{prop: main prop} implies that the conclusions of Proposition \ref{proposition: maximal signature} hold for $\exp(i \theta_0) = \zeta_n$. In particular, Proposition \ref{proposition: maximal signature}(4) shows that all the roots of $\Delta_K(t)$ are contained in the  interval $I_+(\zeta_n)$ while $|\sigma_L(\zeta)| = 2g(L) + (m-1) = \hbox{{\rm deg}}(\Delta_L(t))$ for all $\zeta \in \bar{I}_-(\zeta_n)$. Thus $L$ is definite. This is assertion (1) of Theorem \ref{thm: sqp links}. Assertion (2) of the theorem is Proposition \ref{proposition: maximal signature}(2).  

Given a definite link of $m$ components $L'$ in the $3$-sphere whose Alexander polynomial is not a multiple of $(t-1)^{2g(L) + (m-1)}$, set
$$n_3(L') = \max\{r \geq 2: \Delta_{L'}^{-1}(0) \subset I_+(\zeta_r)\}$$
The integer $n_3(L')$ is well-defined by Proposition \ref{proposition: maximal signature}. It is clear that $n_3(L)$ depends only on $\Delta_{L}(t)$ and from part (1) that $\Sigma_k(L)$ is not an L-space for $k > n_3(L)$. Thus (3) holds. 

Finally we deal with Assertion (4) of the theorem. Suppose that $\Delta_L(t)$ is monic but not $(t-1)^{2g(L) + (m-1)}$. By Proposition \ref{proposition: maximal signature}(5), $\Delta_L(t)$ is a non-trivial product of cyclotomics. In particular, the case that $n = 2$ holds so we assume below that $n \geq 3$. As $\Delta_L(t)$ is not $(t-1)^{2g(L) + (m-1)}$, it is divisible by $\Phi_a(t)$ for some $a > 1$. Since $\Delta_L^{-1}(0) \subset I_+(\zeta_n)$, $a > n \geq 3$. 

If $a = 2j + 1 \geq 3$ is odd, then $j \geq 1$ is relatively prime to $a$ and therefore $\zeta =  \exp(2 \pi i j/a)$ is a primitive $a^{th}$ root of unity. Hence $\zeta \in I_+(\zeta_n)$ so that $j/(2j+1) = j/a < 1/n \leq 1/3$. But this implies that $j < 1$, contrary to our hypotheses.

If $a = 4j \geq 4$ is multiple of $4$, then $\zeta = \exp(2 \pi i (2j-1)/a)$ is a primitive $a^{th}$ root of unity. Hence $\zeta \in I_+(\zeta_n)$ so that $1/2 - 1/a =  (2j-1)/a < 1/n \leq 1/3$. It follows that $a < 6$ and hence, $4 = a > n$. 

If $a = 4j+2 > 1$, it is at least $6$ since $a > n$. Then $\zeta = \exp(2 \pi i (2j-1)/a)$ is a primitive $a^{th}$ root of unity and as above we have $(2j-1)/a < 1/n$. In other words, $n < \frac{a}{2j-1} = \frac{a}{(\frac{a-4}{2})} = 2 + \frac{8}{a-4}$. This implies that $n = 2$ if $a \geq 14$, $n \leq 3$ if $a = 10$ and $n \leq 5$ if $a = 6$. 

We conclude from the last three paragraphs that $n \leq 5$. Further, if $n = 4$ or $5$ then $\Delta_L(t)$ is a product of powers of $\Phi_1$ and $\Phi_{6}$, and if $n = 3$ then $\Delta_L(t)$ is a product of powers of $\Phi_1, \Phi_4, \Phi_6$ and $\Phi_{10}$. Assertion (4) of the theorem now follows.
\end{proof}

\begin{proof}[Proof of Corollary \ref{cor: monic sqp knots}]  
Corollary \ref{cor: monic sqp knots} is an immediate consequence of Theorem \ref{thm: sqp links} once we observe that neither $1$ nor prime power roots of unity are roots of the Alexander polynomial of a knot $K$. In particular, neither $\Phi_1(t)$ nor $\Phi_4(t)$ can be a factor of $\Delta_K(t)$. 
\end{proof}

Next we consider strongly quasipositive satellite knots.  

\begin{prop} \label{prop: satellite} 
Suppose that $K = P(C)$ is a strongly quasipositive satellite knot with non-trivial companion $C$ and pattern $P$ of winding number $w$. Let $K_1 = P(U)$ where $U$ is the unknot.  

$(a)$ If $\vert w \vert \geq 2$, then $\Sigma_n(K)$ is not an L-space for $n \geq 2$. 

$(b)$ If $|w| = 1$ and $\Sigma_n(K)$ is an L-space for some $n \geq 2$, then $|\sigma_{C}(\zeta)| = 2g(C)$ and $|\sigma_{K_1}(\zeta)| = 2g(K_1)$ for all $\zeta \in \bar{I}_-(\zeta_n)$. Thus, both $C$ and $K_1$ are definite. 

$(c)$ If $|w| = 0$ and $\Sigma_n(K)$ is an L-space for some $n \geq 2$, then $g(K_1) = g(K)$. 
\end{prop}

\begin{proof} 
Since $K$ is a strongly quasipositive knot if $\Sigma_n(K)$ is an L-space for some $n \geq 2$, Proposition \ref{prop: main prop} implies that $|\sigma_K(\zeta_n^j)| = 2g(K)$ for $1 \leq j \leq n-1$. By \cite{Sch}, one has $g(K) \geq \vert w \vert g(C) + g(K_1)$, and by \cite[Theorem 2]{Lith}, $\sigma_K(\zeta_n^j) = \sigma_{C}(\zeta_n^{wj}) + \sigma_{K_1}(\zeta_n^j)$ for each $1 \leq j \leq n-1$. Then $|\sigma_K(\zeta_n^j)| = 2g(K) \geq 2 |w| g(C) + 2g(K_1) \geq |w| |\sigma_{C}(\zeta_n^{wj})| + |\sigma_{K_1}(\zeta_n^j)| \geq |\sigma_{C}(\zeta_n^{wj})| + |\sigma_{K_1}(\zeta_n^j)| \geq |\sigma_K(\zeta_n^j)|$. Therefore this sequence of inequalities is a sequence of equalities. If $w \ne 0$, this can happen only if $|w| = 1$,  $|\sigma_{C}(\zeta_n^j)| = 2g(C)$ and $|\sigma_{K_1}(\zeta_n^j)| = 2g(K_1)$ for $1 \leq j \leq n-1$. Proposition \ref{proposition: maximal signature} now implies that (a) and (b) hold. If $w = 0$, we have $g(K_1) \leq 2g(K) = |\sigma_K(\zeta_n^j)| = |\sigma_{K_1}(\zeta_n^j)| \leq 2g(K_1)$, so $g(K_1) = g(K)$, which  completes the proof.
\end{proof}

\begin{rem} {\rm $(1)$ A classical argument of Schubert implies that case $(c)$ of Proposition \ref{prop: satellite} does not arise when $K$ is a {\it fibred} strongly quasipositive satellite knot. Indeed, Schubert showed that the existence of a companion with zero winding number would imply that the complement of the fibre surface of $K$ contains an essential torus. But this is impossible since the complement has a free fundamental group.

$(2)$ If $K$ is the closure of a strongly quasipositive braid $\beta$ and $\Sigma_n(K)$ is an L-space for some $n \geq 2$, then case $(a)$ of Proposition \ref{prop: satellite} implies that $\beta$ is either pseudo-Anosov or periodic as a diffeomorphism of the disk with holes. In the latter case, $K$ is a torus knot.

$(3)$ Ken Baker has constructed examples of prime fibred strongly quasipositive satellite knots for which case $(b)$ occurs. His work with Motegi shows that it cannot arise for L-space knots.}
\end{rem}

\begin{cor} \label{cor: lspace knot satellite}
If $K$ is a satellite L-space knot, then no $\Sigma_n(K)$ is an L-space.
\end{cor}

\begin{proof}
Baker and Motegi show that satellite L-space knots can be expressed as a satellite knot where the pattern $P$ is a braid (\cite[Theorem 7.4]{BMot}). In this case, the winding number of $P$ is at least two in absolute value, so Proposition \ref{prop: satellite}(a) implies that no $\Sigma_n(K)$ is an L-space.
\end{proof}

\begin{proof}[Proof of Corollary \ref{cor: lspace knots}]  
If $3 \leq n \leq 5$, Corollary \ref{cor: monic sqp knots} implies that $\Delta_K(t)$ is of the form $\Phi_6^i \Phi_{10}^j$ for some $i + j > 0$. An elementary computation shows that if such a product yields a polynomial whose non-zero coefficients are $\pm 1$, then $i + j = 1$. Hence as the Alexander polynomials of L-space knots satisfy this condition, Corollary \ref{cor: monic sqp knots} implies that $\Delta_K(t)$ is $\Phi_6(t)$ if $n = 4, 5$ and either $\Phi_6(t)$ or $\Phi_{10}(t)$ if $n = 3$. The $(2,3)$ torus knot is the only fibred knot with Alexander polynomial $\Phi_6(t)$, so we deduce part (1) of the corollary. The only torus knot with Alexander polynomial $\Phi_{10}(t)$ is the $(2, 5)$ torus knot, so part (2) of the corollary will follow if we can show that if $K$ is a satellite knot with Alexander polynomial $\Phi_{10}(t)$, then $\Sigma_3(K)$ is not an L-space. This follows from Corollary \ref{cor: lspace knot satellite}. It also follows from \cite{HRW}, since the $3$-fold branched cover of any satellite knot with Alexander polynomial $\Phi_{10}(t)$ is an integer homology $3$-sphere. It is also toroidal by  \cite[Theorem 2]{GLit}. But then it is not an L-space by \cite[Corollary 9]{HRW}. 
\end{proof}

\section{Strongly quasipositive alternating links} \label{sec: sqp alt}

The following proposition follows from work of Kunio Murasugi. 

\begin{prop} \label{prop: definite alternating} 
Let $L$ be a non-split alternating link with $m$ components, then $L$ is definite if and only if $L$ is special alternating. 
\end{prop}

\begin{proof} 
If $L$ is special alternating, its signature is maximal by \cite[Lemma 5.2]{Mu2} and \cite[Theorem 3.8]{Mu1}. 

Conversely, suppose that $L$ is a definite alternating link. According to Murasugi \cite[Theorem 5.4]{Mu2}, $L$ is the star-product of finitely many special alternating links and the signature of $L$ is the sum of the signatures of the special alternating factors. This star product representation is obtained by applying the Seifert algorithm to a minimal alternating projection $D$ of $L$. The resulting Seifert surface $S$ of $L$ naturally decomposes as a Murasugi plumbing of surfaces obtained by applying the Seifert algorithm to alternating projections of the special alternating factors obtained from $D$. The crossings of each of the factors are either all positive or all negative, and adjacent factors have crossings of opposite sign. In particular, the signatures of adjacent factors are of opposite sign. But $L$ has maximal signature and $D$ is minimal, so the symmetrised Seifert form on $S$ is definite, which implies the symmetrised Seifert forms on the factors are definite of the same sign. Consequently, there can be only one factor, and therefore $L$ is special alternating. 
\end{proof}

Theorem 3.24 of \cite{Mu1} states that the only special alternating knots with monic Alexander polynomials are connected sums of $(2, k)$ torus knots ($k$ is allowed to vary). Hence, we have the following corollary of Proposition \ref{prop: definite alternating}. 

\begin{cor}
A prime fibred alternating knot is definite if and only if it is a $(2, k)$ torus knot. 
\qed 
\end{cor}

Special alternating links are positive, hence strongly quasipositive by \cite{Ru4}, so our next corollary follows by combining Proposition \ref{prop: main prop} and the fact that the 2-fold branched cover of an alternating link is an L-space (\cite[Proposition 3.3]{OS2}). For knots, it also follows from a previous result of S. Baader (\cite{Baa3}). 

\begin{cor} \label{cor:sqp alternating} 
An alternating link $L$ is strongly quasipositive if and only if it is special alternating. A prime, fibred alternating knot is strongly quasipositive if and only if it is a $(2, k)$ torus knot. 
\qed
\end{cor}

The next corollary is a consequence of Proposition \ref{prop: definite alternating}, Theorem \ref{thm: sqp links}, and Theorem \ref{thm: qp links}(1) (which is proved in \S \ref{sec: qp links}, independently of the corollary). It answers positively a question of S. Baader (\cite[Question 3, page 268]{Baa1}) in the case of alternating links.

\begin{cor} \label{cor:qp alternating} A non-split quasipositive alternating link $L$  is strongly quasipositive if and only if $g(L) = g_4(L)$. 
\end{cor}

\begin{proof} If $L$ is strongly quasipositive, then $L$ is quasipositive and $g_4(L) = g(L)$ by Theorem \ref{thm: sqp links}. (Recall that $\Sigma_2(L)$ is an L-space as $L$ is a non-split alternating link.) Conversely if $L$ is quasipositive, $|\sigma(L)| = 2g_4(L) + (m-1)$
by Theorem \ref{thm: qp links}(1). Therefore if $g_4(L) = g(L)$, $L$ is definite and therefore strongly quasipositive by Proposition \ref{prop: definite alternating}.
\end{proof}

In \cite[Proposition 3.6 and Corollary 3.7]{BR}, it is proved for a family of alternating arborescent links, including  2-bridge links, that quasipositivity implies strong  quasipositivity. So, in view of Corollary \ref{cor:qp alternating}, we ask the following question.

\begin{que} 
{\rm Does there exist a quasipositive alternating link which is not strongly quasipositive?}
\end{que}

\section{Strongly quasipositive pretzel links} \label{sec: sqp pretzels}

In this section we apply the results of \S \ref{sec: sqp} and \S \ref{sec: sqp alt} to study the strong quasipositivity of pretzel links. First we settle the question of which alternating pretzel links are strongly quasipositive, using Corollary \ref{cor:sqp alternating}. By a {\it projective orientation} on a link $L$ we mean an orientation modulo reversal of the orientations on all components of $L$. 

\begin{prop} \label{prop: alt pretzel links}
Let $L$ be the pretzel link $L(p_1,...,p_n)$ where $n \ge 3$, and $p_i \ge 2, 1 \le i \le n$.

$(1)$ If all the $p_i$ are odd then, with any orientation, $L$ is strongly quasipositive.

$(2)$ If all the $p_i$ are even then $L$ has an orientation with respect to which it is strongly quasipositive. Moreover, there is exactly one such projective orientation if $n$ is odd and exactly two if $n$ is even.

$(3)$ If there exist both odd and even $p_i$ then $L$ has an orientation with respect to which it is strongly quasipositive if and only if $n$ is even. Moreover, this orientation is projectively unique.
\end{prop}

\begin{proof}
By hypothesis, $L$ is alternating and the obvious diagram for $L$ is reduced and therefore minimal. 

(1) If $n$ is odd then $L$ has one component, and for any orientation of $L$ all the crossings in the obvious diagram are positive. If $n$ is even then $L$ has two components, and for any orientations of these components, either all crossings are positive or all crossings are negative. Hence, in all cases, $L$ is strongly quasipositive by Corollary 7.3.

(2) In both cases $n$ odd and $n$ even there is a unique projective orientation of $L$ that makes the obvious diagram negative. If $n$ is even there is a unique projective orientation making the diagram positive, and no such orientation if $n$ is odd.

(3) Let $k$ be the number of even $p_i$; so $0 < k < n$. Suppose they occur in (cyclic) order $e_1,...,e_k$. Let $m_j$ be the number of (odd) $p_i$ strictly between $e_j$ and $e_{j+1}$ (interpreted cyclically). Note that $m_j$ may be zero, but there is at least one $m_j$ that is strictly positive. The $m_j$ determine the components $K_1,...,K_k$ of $L$, where $K_j$ is a pretzel knot with $m_j$ strands. For any orientation of $K_j$, the self-crossings of $K_j$ are all positive. Hence, if $L$ with some orientation is strongly quasipositive then by Corollary 7.3 all crossings must be positive. Choosing an orientation on $K_1$ then determines an orientation on $K_2$, and so on. It is easy to see that this determines a consistent orientation on $L$ if and only if the number of even $m_j$ is even. Since $ \displaystyle n = \sum_{j=1}^{k} (m _j + 1)$, this is equivalent to the condition that $n$ be even.
\end{proof} 

We now restrict our attention to knots. Let $K = P(\epsilon_1 p_1,...,\epsilon_n p_n)$ be a pretzel knot with $n \ge 3$ strands, $p_i \ge 2$ and $\epsilon_i = \pm 1, 1 \le i \le n$. Since $K$ is a knot, either exactly one $p_i$ is even, or $n$ and each $p_i$ is odd. In the first case we may assume that $p_1$ is even. In this case we will determine exactly when $K$ is strongly quasipositive (Proposition 8.5). In the second case we are only able to obtain a partial result (Proposition 8.6).

Let $d = \displaystyle \sum_{i=1}^n \epsilon_i $; so $d \equiv n$ (mod $2$).

From now until the end of Conjecture \ref{conj: odd case} we will assume the above notation.

The proof of Proposition 8.5 will use the following lemma, asserting that under certain conditions $g_4(K) < g(K)$, which implies that K is not strongly quasipositive (cf. \S \ref{sec: positivity}). 

\begin{lemma} \label{lemma: pretzel genera} 
If $p_1$ is even and either 
\vspace{-.25cm} 
\begin{itemize}
\item there exist $i, j$, $2 \leq i < j \leq n$ such that $\varepsilon_i \varepsilon_j = -1$, or

\vspace{.1cm} \item $n$ is even and there exists $i$, $2 \leq i \leq n$, such that $\varepsilon_1 \varepsilon_i = -1$,
\end{itemize}
\vspace{-.3cm} 
then $g_4(K) < g(K)$.

\end{lemma}

\begin{proof}
Assume first that there exist $i, j$, $2 \leq i < j \leq n$ such that $\varepsilon_i \varepsilon_j = -1$. Then there is an $i$,$1 < i < n$, so that $\varepsilon_i \varepsilon_{i+1} = -1$. With some orientation of $K$, the $i^{th}$ and $(i+1)^{st}$ (pairs of) strands are oriented as in Figure \ref{fig: g4 < g 1}, which illustrates the case $\varepsilon_i = 1, \varepsilon_{i+1} = -1$. 

\begin{figure}[!ht]
\centering

 \includegraphics[scale=0.7]{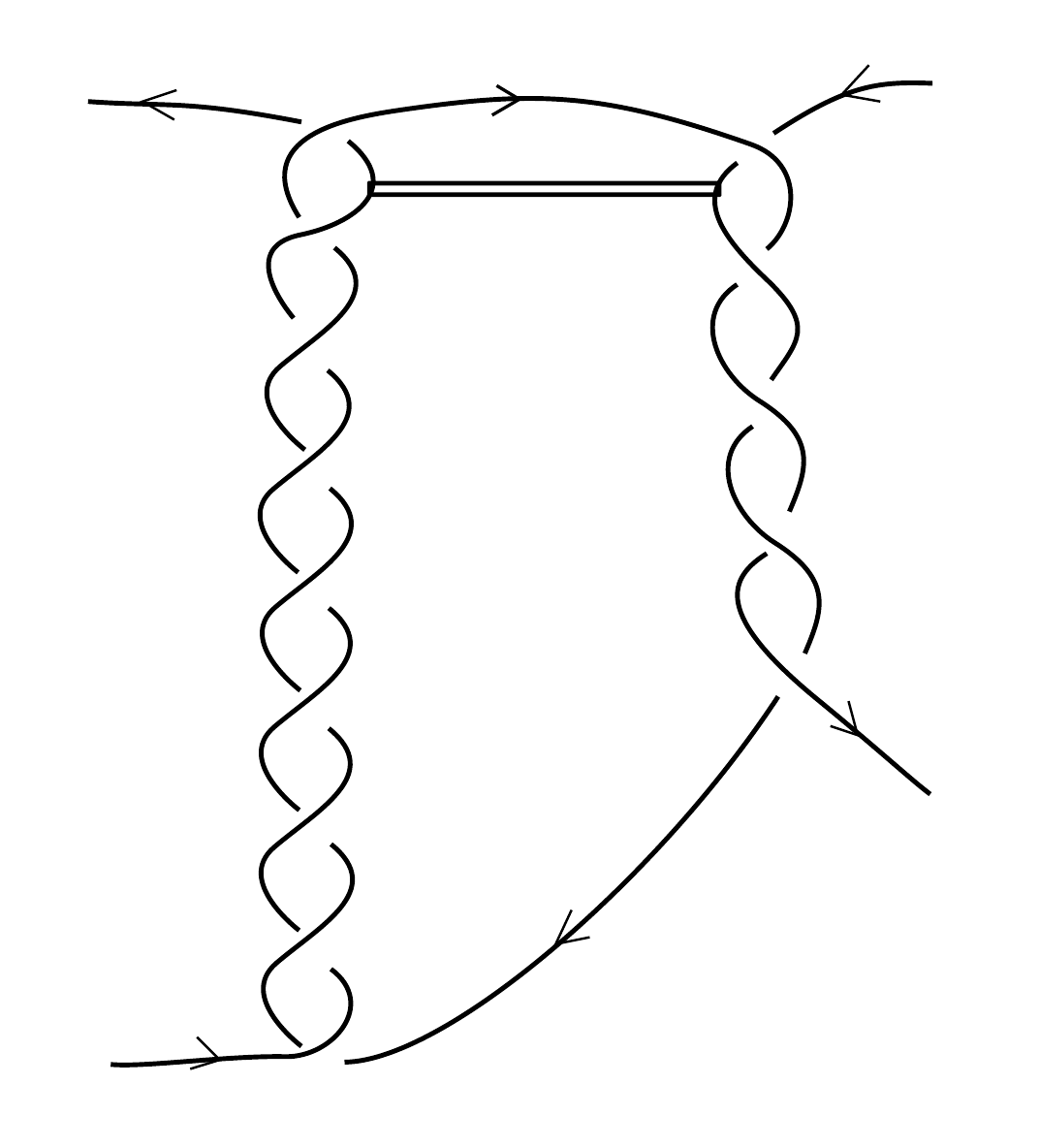}
 \caption{} 
 \label{fig: g4 < g 1}
\end{figure} 

Performing the band move shown in Figure \ref{fig: g4 < g 1} gives a thrice-punctured $2$-sphere $Q$ in $S^3 \times I$ such that $Q \cap ( S^3 \times \{1\}) = K$ and $Q \cap ( S^3 \times \{0\}) = K_0 \# T(2,k)$, where $k = \varepsilon_i p_i + \varepsilon_{i+1} p_{i+1}$ and $K_0$ is the knot $P(\varepsilon_1 p_1, \ldots, \varepsilon_{i-1} p_{i-1}, \varepsilon_{i+2} p_{i+2}, \ldots, \varepsilon_n p_n)$. Note that the components of the torus link $T(2,k)$ are oriented coherently as they lie on the torus. 

If $k \ne 0$, $T(2,k)$ has a Seifert surface $S_k$ of genus $|k|/2 - 1 = (|p_i - p_{i+1}| - 2)/2$. If $k = 0$, $T(2,0)$ is the $2$-component unlink, which bounds a pair of disjoint disks, $S_0$.

In $S^3 \times \{0\}$, $K_0 \# T(2,k)$ bounds the boundary connected sum $F$ of $S_k$ with a minimal genus Seifert surface for $K_0$. Taking the union of this surface $F$ with $Q$ shows that 
\begin{equation}
g_4(K) \leq g(K_0) + \frac12 (|p_i - p_{i+1}| - 2) + 1 = g(K_0) + \frac{|p_i - p_{i+1}|}{2} 
\end{equation}
On the other hand, by \cite[Theorem 4.1]{KL}, 
\begin{equation}
g(K) = g(K_0) + \frac12 (p_i + p_{i+1} - 2) = g(K_0) + \frac{(p_i + p_{i+1})}{2} - 1 
\end{equation}
Therefore, 
$$g(K) - g_4(K) \geq \frac12 \big((p_i + p_{i+1}) - |p_i - p_{i+1}|\big) - 1 = \min \{p_i, p_{i+1}\} - 1 \geq 2$$
Hence we are done.

Next suppose that $n$ is even and there exists an $i$, $2 \leq i \leq n$, such that $\varepsilon_1 \varepsilon_i = -1$. Without loss of generality we can suppose that $\varepsilon_1 \varepsilon_2 = -1$ by what we have just proven. 

The reader will verify that under our assumptions, the orientation on the first two (pairs of) strands of $K$ are as shown in Figure \ref{fig: g4 < g 1}. So we may perform the same band move as in the first case to obtain a thrice-punctured $2$-sphere $Q$ in $S^3 \times I$ such that $Q \cap ( S^3 \times \{1\}) = K$ and $Q \cap ( S^3 \times \{0\}) = L_0 \# T(2,k)$, where $k = \varepsilon_1 p_1 + \varepsilon_{2} p_{2}$ and $L_0$ is the $2$-component link $P(\varepsilon_3 p_3, \ldots, \varepsilon_n p_n)$.

The induced orientation on $L_0$ is such that each (pair of) strands is coherently oriented. Then the Seifert circles construction yields a Seifert surface for $L_0$ of genus $\displaystyle \frac12 \sum_{i=3}^{n} (p_i - 1)$. The genus of $T(2,k)$ is $(|k| - 1)/2 = (|p_1 - p_2| - 1)/2$ since $\varepsilon_1 \varepsilon_2 = -1$. Hence, as in the first case, we see that 
\begin{equation}
g_4(K) \leq \displaystyle \frac12 \big(\sum_{i=3}^{n} (p_i - 1) + (|p_1 - p_2| - 1 ) \big) + 1
\end{equation}
On the other hand, referring to the case $n$ even and ``$\alpha + \beta \ne 0$" in \cite[Theorem 4.1]{KL}, we have 
\begin{equation}
\displaystyle g(K) = \frac12 \big(p_1 +  \sum_{i=2}^{n} (p_i - 1) \big)
\end{equation}
Therefore, 
\begin{eqnarray} 
g(K) - g_4(K) & \geq & \frac12 \big((p_1 + p_2) - |p_1 - p_{2}|\big) - 1 \nonumber \\ 
& = & \min \{p_1, p_{2}\} - 1 \nonumber \\ 
& \geq & 1 \nonumber
\end{eqnarray} 
which completes the proof.
\end{proof}

\begin{prop} \label{prop: pretzel knot sqp} 
Suppose that $p_1$ is even. Then $K$ is strongly quasipositive if and only if either

$(1)$ $|d| = n$ and $n$ is even; or

$(2)$ $\vert d \vert = n -2$, $n$ is odd, and $\varepsilon_2 = \varepsilon_3 = \ldots = \varepsilon_n$.
\end{prop} 

\begin{proof} 
If $|d| = n$, then $K$ is alternating so $K$ is strongly quasipositive if and only if $n$ is even by Proposition \ref{prop: alt pretzel links}(3). 

If $|d| \leq n-4$ then the first of the two conditions listed in Lemma \ref{lemma: pretzel genera} applies to show that $K$ is not strongly quasipositive. 

Finally suppose that $|d| = n-2$. Then the first of the two conditions listed in Lemma \ref{lemma: pretzel genera} shows that, unless $\varepsilon_1 \varepsilon_i = -1$, $2 \leq i \leq n$, we have $g_4(K) < g(K)$ and so $K$ is not strongly quasipositive. Assume then that $\varepsilon_1 \varepsilon_i = -1$ for $2 \leq i \leq n$. If $n$ is even, then $K$ is not strongly quasipositive by the second of the two conditions listed in Lemma \ref{lemma: pretzel genera}. If $n$ is odd, then the obvious knot diagram of $K$ is positive, and therefore $K$ is strongly quasipositive. 
\end{proof}

As a consequence of the proof we obtain: 

\begin{cor} \label{cor: g versus g4}
If $p_1$ is even, then $K$ is strongly quasipositive if and only if $g_4(K) = g(K)$.
\qed
\end{cor}

The corollary does not hold for pretzel knots in general, even for classical pretzel knots. See Remark \ref{rem: g versus g4}. On the other hand, given Baader's question \cite[Question 3, page 268]{Baa1}, we ask: 

\begin{question}
If $K$ is a quasipositive pretzel knot with $g_4(K) = g(K)$, is $K$ strongly quasipositive? 
\end{question}

We now consider the case where all the $p_i$ are odd. We begin with a lemma which determines precisely which pretzel links have L-space $2$-fold branched covers. Throughout we assume that our pretzel links are non-split. Since pretzel links with two or fewer strands are $2$-bridge and hence have L-space $2$-fold branched covers, we restrict our attention to those with three or more strands. 

\begin{lemma} \label{lemma: pretzel l-space cover} 
Let $L$ be the pretzel link $P(p_1, \ldots, p_n, -q_1, \ldots, -q_r)$ where $n \geq r$, $n+r \geq 3$, $p_n \geq p_{n-1} \geq \ldots \geq p_1 \geq 2$, and $q_j \geq 2$ for each $j$. Then $\Sigma_2(L)$ is an L-space if and only if $($i$)$ $r = 0$ or $($ii$)$ $r = 1$ and either $q_1 \geq p_1$ or $q_1 = p_1 - 1$ and $2q_1+1 \geq p_2$. 
\end{lemma}

\begin{remark}
{\rm Since the homeomorphism type of $\Sigma_2(L)$ is invariant under mirroring $L$, permuting $p_1, \ldots, p_n, -q_1, \ldots, -q_r$, or altering the orientations of the components of $L$, the lemma gives necessary and sufficient conditions for an arbitrary pretzel link on three or more strands to have an L-space $2$-fold branched cover.}
\end{remark}

\begin{proof}[Proof of Lemma \ref{lemma: pretzel l-space cover}]
The $2$-fold branched cover of $L$ is the Seifert manifold
$$\Sigma_2(L) = M(r; \frac{1}{p_1}, \ldots, \frac{1}{p_n}, \frac{q_1 - 1}{q_1} \ldots, \frac{q_r-1}{q_r})$$
(See e.g. \cite[Proposition 3.3]{OwSt}.) 

If $r \geq 2$, then $n \geq 2$ so that $2 \leq r \leq n+r -2$. Then \cite[Theorem 2 and \S 7]{JN} implies that $\Sigma_2(L)$ admits a co-oriented taut foliation, hence is not an L-space. 

If $r = 0$, then $L$ is a non-split alternating link, so $\Sigma_2(L)$ is an L-space (\cite[Proposition 3.3]{OS2}).

Assume then that $r = 1$ and set $q = q_1$. If $q \leq p_1 - 2$ (so $p_1 \geq 4$) set $m = p_1 - 1$ and $a = 1$. Then 
$\frac{1}{p_i} < \frac{1}{m}$ for each $i$ while $\frac{1}{m} < \frac{1}{q}$, which implies that $\frac{q-1}{q} < \frac{m-1}{m}$. By Naimi's completion of the proof of \cite[Conjecture 2]{JN}, see \cite{Nai}, $\Sigma_2(L)$ admits a co-oriented taut foliation, so is not an L-space. 

Next suppose that $q \geq p_1 - 1$. By \cite{JN} and \cite{Nai}, $\Sigma_2(L)$ admits a co-oriented taut foliation if and only if there is a coprime pair $a, m$ such that $0 < a \leq \frac{m}{2}$ and a permutation $(\frac{a_1}{m}, \frac{a_2}{m}, \ldots, \frac{a_{n+1}}{m})$ of $(\frac{a}{m}, \frac{m-a}{m}, \frac{1}{m}, \ldots, \frac{1}{m})$ such that $\frac{1}{p_i} < \frac{a_i}{m}$ for $1 \leq i \leq n$ and $\frac{q-1}{q} < \frac{a_{n+1}}{m}$. Since $\frac{q-1}{q} \geq \frac{1}{2}$, it must be that $a_{n+1} = m - a$ where $m \geq 3$. Thus $\frac{a}{m} < \frac{1}{q}$. Since $\frac{1}{p_n} \leq \frac{1}{p_{n-1}} \leq \ldots \leq \frac{1}{p_1}$, we can suppose that $a_1 = a$ and $a_i = 1$ for $2 \leq i \leq n$. Thus $\Sigma_2(L)$ admits a co-oriented taut foliation if and only if there is a coprime pair $a, m$ such that 
$$\frac{1}{p_1} < \frac{a}{m} < \frac{1}{q} \hbox{ and } \frac{1}{p_i} < \frac{1}{m} \hbox{ for } 2 \leq i \leq n$$
Equivalently, 
$$q < \frac{m}{a} < p_1 \hbox{ and } m < p_2$$
This is impossible for $q \geq p_1$, so we must have $q = p_1 -1$. (Thus $\Sigma_2(L)$ is an L-space if $q \geq p_1$.) The inequality $q < \frac{m}{a} < p_1 = q+1$ implies that $a \geq 2$.  Then $2q + 1 < 2(\frac{m}{a}) + 1 \leq m + 1 \leq p_2$. Hence $\Sigma_2(L)$ is not an L-space if $q = p_1 -1$ and $2q+1 < p_2$. Conversely if $q = p_1 -1$ and $2q+1 < p_2$, take $m = 2q+1$ and $a = 2$. Then $q < \frac{m}{2} < q+1 = p_1$ and $\frac{1}{p_i} < \frac{1}{m} \hbox{ for } 2 \leq i \leq n$. Thus $\Sigma_2(L)$ admits a co-oriented taut foliation, which  completes the proof.
\end{proof}

\begin{prop} \label{prop: pi all odd} 
Suppose that $p_i$ is odd, $1 \le i \le n$.

$(1)$ If $|d| = n$ then $K$ is strongly quasipositive.

$(2)$ If $|d| = n-2$ $($so without loss of generality $\epsilon_1 = ... = \epsilon_{n-1} = +1, \epsilon_n = -1$ $)$, then $K$ is strongly quasipositive if and only if $p_n < min\{p_1,...,p_{n-1}\}$.

\end{prop} 

\begin{proof} 
(1) This follows from Proposition \ref{prop: alt pretzel links}(1).

(2) We know from \cite[Theorem]{Ru5} that $K$ is strongly quasipositive if $p_n < \min\{p_1, p_2,  \ldots, p_{n-1}\}$. Conversely if $p_n \geq \min\{p_1, p_2,  \ldots, p_{n-1}\}$, Lemma \ref{lemma: pretzel l-space cover} implies that $\Sigma_2(K)$ is an L-space. If $K$ is strongly quasipositive, Theorem \ref{thm: sqp links} implies that $|\sigma(K)| = 2g(K)$. On the other hand, Jabuka has calculated the signature of $K$ (\cite[Theorem 1.18(3)]{Ja}): 
$$\sigma(K) = (n -2) + \hbox{sign}\big(\frac{1}{p_n} - (\frac{1}{p_1} + \frac{1}{p_2} + \ldots + \frac{1}{p_{n-1}})\big)$$ 
and since $p_n \geq \min\{p_1, p_2,  \ldots, p_{n-1}\}$ and $n \geq 3$, $\hbox{sign}\big(\frac{1}{p_n} - (\frac{1}{p_1} + \frac{1}{p_2} + \ldots + \frac{1}{p_{n-1}})\big) = -1$. Thus $\sigma(K) = n-3$. 
On the other hand, Gabai has shown that $2g(K) = n-1$ (\cite[\S 3]{Ga3}). Thus $|\sigma(K)| < 2g(K)$, so $K$ is not strongly quasipositive. 
\end{proof} 

\begin{rem}  
{\rm Let $F$ be the obvious (minimal genus) checkerboard Seifert surface for $K$, with $2g(K) = n-1$. If $|d| \le n-4$ then $F$ is not quasipositive. If we knew that $F$ was the unique minimal genus Seifert surface for $K$ then we could conclude that $K$ is not strongly quasipositive.

Another condition under which we can show that $K$ is not strongly quasipositive is if some non-empty subset of $\{\epsilon_1 p_1,...,\epsilon_n p_n\}$ has sum zero. For then it is easy to see that there is a non-separating simple closed curve on $F$ with framing zero. Doing surgery on this curve in the 4-ball shows that $g_4(K) < g(K)$, and hence $K$ is not strongly quasipositive.}
\end{rem}

In light of these observations we make the following conjecture.

\begin{conj} \label{conj: odd case}
{\rm If all the $p_i$ are odd then $K$ is strongly quasipositive if and only if either $|d| = n$, or $|d| = n-2$ and the condition in Proposition \ref{prop: pi all odd}(2) holds.}
\end{conj}

If $n = 3$ then (1) and (2) of Proposition \ref{prop: pi all odd} are the only possibilities, so we can prove Corollary \ref{cor: 3-strand knots}.

\begin{proof}[Proof of Corollary \ref{cor: 3-strand knots}] 
The case $r > 0$ follows from parts (1) of Propositions \ref{prop: pretzel knot sqp} and \ref{prop: pi all odd}, and the case $r < 0$ from parts (2) of these propositions. 
\end{proof} 

\begin{rem} \label{rem: g versus g4} 
{\rm The analogue of Corollary \ref{cor: g versus g4} does not hold for pretzel knots all of whose parameters are odd. For instance,  it follows from \cite[Theorem 1.5]{Mill}  
that the pretzel knot $K = P(3, 5, -7)$ with non-trivial Alexander polynomial verifies $g_4^{top}(K) = g_4(K) = g(K) = 1$ while it is not strongly quasipositive by Corollary \ref{cor: 3-strand knots}}.
\end{rem}

We can also classify the strongly quasipositive $3$-strand pretzel links, though the necessity of considering different orientations increases the number of cases to be examined. 

Let $L = P(p, q, r)$ be a pretzel link where the parameters $p, q, r$ are at least $2$ in absolute value. Up to taking a mirror image and cyclically permuting the parameters, we can suppose that $p, q \geq 2$, and given that we have just dealt with the case that $L$ is a knot, we assume that $L$ has two or three components. Thus at least two of the parameters are even, and without loss of generality we can assume that $p$ is even.  

Up to a simultaneous change in the orientations of the components of $L$, the four possibilities for the orientation of $L$ are depicted in Figure \ref{fig: pretzel orientations}. 

\begin{figure}[!ht]
\centering
 \includegraphics[scale=0.9]{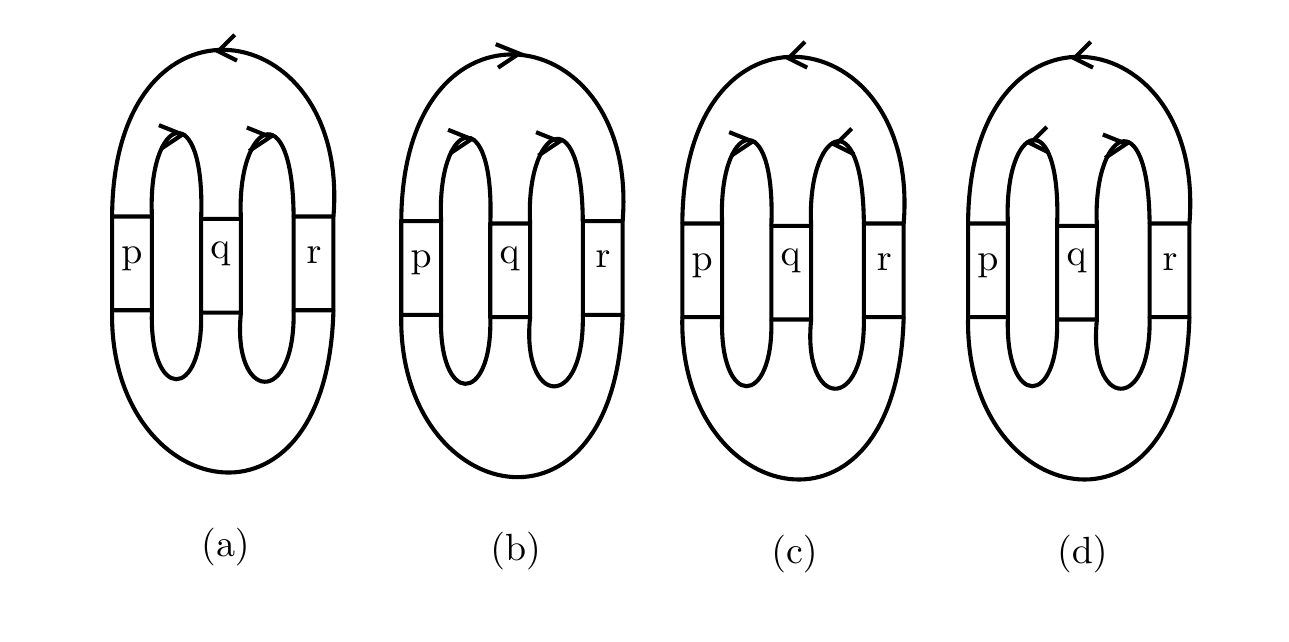}
  \caption{Possible orientations for $L = P(p, q, r)$} 
 \label{fig: pretzel orientations}
\end{figure} 

With these conventions, we remark that if $L = P(p, q, r)$ carries the orientation $(a)$, then $p, q$ and $r$ must be even. If $L$ carries the orientation $(b)$, then $q$ must be even,
and if $L$ carries the orientation $(d)$, $r$ must be even.

\begin{prop} \label{prop:3-strand links}
A $3$-strand oriented pretzel link $L = P(p, q, r)$ with $p, q, |r| \geq 2$ and $p$ even is strongly quasipositive if and only if $(p, q, r)$ verifies one of the following conditions: 

$(i)$ $r > 0$, each of $p, q , r$ is even, and $L$ carries the orientation $(a)$; 

$(ii)$ $r <0$ and $L$ carries the orientation $(d)$; 

$(iii)$ $r <0$, $|r| < min(p, q)$, and $L$ carries the orientation $(a)$.

In particular, if $L$ carries the orientation $(b)$ or $(c)$, it is never strongly quasipositive.
 
\end{prop}

\begin{proof}
In the case that $r >0$, $L$ is alternating, so (i) follows from Corollary \ref{cor:sqp alternating}, as in the proof of Proposition \ref{prop: alt pretzel links}(2). 

Next suppose that $r < 0$. We consider the different possible orientations for $L$ one-by-one.  

If $L$ carries orientation $(a)$, then $p, q$ and $r$ must be even and hence the braid surface $F$ consisting of the upper and lower disks in the diagram connected by the three pretzel bands is a Seifert surface for $L$. It then follows from \cite[Lemma 3]{Ru3} that $L$ is strongly quasipositive as long as $|r| < \min(p, q)$. On the other hand, if $|r| \geq \min(p, q)$, then $\Sigma_2(L)$ is an L-space by Lemma \ref{lemma: pretzel l-space cover}. The reader will verify that $F$ has a Seifert matrix of the form $\frac12 \left(\begin{smallmatrix} p+r  &   r \\ r   &  q + r \end{smallmatrix}\right)$ and hence after summing with its transpose we obtain $\left(\begin{smallmatrix} p+r  &   r \\ r   &  q + r \end{smallmatrix}\right)$. The determinant of the latter is $pq + pr + qr = (q+r)p + qr = (p+r)q + pr$. By hypothesis, at least one of $p+r$ and $q+r$ is non-positive and since $pr$ and $qr$ are both negative, the determinant of $\left(\begin{smallmatrix} p+r  &   r \\ r   &  q + r \end{smallmatrix}\right)$ is less than $0$. Hence its signature, which is the signature of $L$, is $0$. It follows that $L$ cannot be strongly quasipositive as otherwise Theorem \ref{thm: sqp links} would imply that $0 = |\sigma(L)| = 2g(L) + (m-1) \geq 2$, a contradiction. 

In what follows, we use the term {\it bamboo} to denote the boundary of a plumbing of Hopf bands, all either positive or negative, along a connected graph whose vertices have at most two neighbours. 

If $L$ carries orientation $(b)$, then $q$ is even. The two non-oriented bands with $p$ and $r$ twists can be blown up using Neumann's plumbing calculus (\cite{Neu}) to show that the (oriented) link $L$ is the oriented boundary of the surface $S$ obtained from the pretzel surface $F(-2, q , 2)$ by plumbing a bamboo of $p-2$ positive Hopf bands on the band labeled $-2$ and a bamboo of $|r| -2$ negative Hopf bands on the band labeled $2$. Since $q$ is even, the pretzel link $P(-2, q, 2)$ is fibred with fibre surface $F(-2, q, 2)$, and thus by \cite{Ga1}, the oriented link $L$ is a fibred link with fibre $S$. By \cite[Theorem 4.5]{Ru5}, the pretzel surface $F(-2, q, 2)$ is not quasipositive (up to mirror image) and thus  by \cite[Thm 2.15]{Ru6}, the fibre surface $S$ is not quasipositive (up to mirror image). 
But then the (oriented) fibred link $L$ is not strongly quasipositive since it bounds a unique fibre surface, up to isotopy.

Suppose that $L$ carries orientation $(c)$. The argument in this case is similar to the previous one. The two non-oriented bands with $q$ and $r$ twists can be blown up using Neumann's plumbing calculus to show that $L$ is the oriented boundary of the surface $S$ obtained from the pretzel surface $F(p, -2, 2)$ by plumbing a bamboo of $q-2$ positive Hopf bands on the band labeled $-2$ and a bamboo of $|r| -2$ negative Hopf bands on the band labeled $2$. 
Since $p$ is even the pretzel link $P(p, -2, 2)$ is fibred with fibre surface $F(p, -2, 2)$, and thus by \cite{Ga1}, $L$ is a fibred link with fibre $S$; As the pretzel surface $F(p, -2, 2)$ is not quasipositive (up to mirror image), we conclude as before that  $L$ is not strongly quasipositive.

Finally suppose that $L$ carries the orientation $(d)$. Then the given diagram is positive and therefore $L$ is a positive link. The main result of \cite{Ru4} now shows that $L$ is strongly quasipositive.
\end{proof}

\begin{problem}
{\rm Determine necessary and sufficient conditions for pretzel links to be strongly quasipositive.}
\end{problem}

\section{Branched covers of fibred strongly quasipositive knots} \label{sec: fsqpm knots}

\begin{prop} \label{prop: fsqpm signature Michel} 
Let $K$ be a strongly quasipositive fibred Montesinos knot, then $K$ is definite if and only if $K$ is a $(2,k)$, $(3,4)$, or $(3,5)$ torus knot. 
\end{prop}

\begin{proof} 
Let $K$ be a strongly quasipositive fibred Montesinos knot. If $K$ has two or fewer branches, it is a $2$-bridge knot. It then follows from \cite{BR} and \cite{Ga2} that it is a bamboo. These are exactly the $(2, k)$ torus knots. So for the remainder of the proof we assume that the number $r$ of branches of $K$ is at least $3$.

According to Baker and Moore \cite[Theorem 2]{BM}, a fibred strongly quasipositive Montesinos knot $K$ with three or more branches is isotopic to either: \begin{itemize}

\item[(a)] $K(\frac{-d_1}{2d_{1} + 1}, \cdots , \frac{-d_r}{2d_{r} + 1} \vert 1)$, $d_1, \cdots , d_r$ a set of positive integers such that $\sum_{i=1}^{r} d_i$ is even, or \\

\item[(b)] $K(\frac{-m_1}{m_{1} + 1}, \cdots , \frac{-m_r}{m_{r} + 1} \vert 2)$, $m_1 \geq 1$ an odd integer and $m_2, \cdots , m_r \geq 2$ even integers.

\end{itemize}

In case $(a)$ the knots are called {\it odd type} Montesinos knots and in case $(b)$ {\it even type} Montesinos knots. We use the plumbing calculus of Neumann \cite{Neu}, or equivalently the Kirby calculus, to describe these knots as the boundaries of plumbings of twisted bands according to a star shaped tree with $r$ branches. 

We deal with cases (a) and (b) separately.

\subparagraph{{\bf Case (a)}}
Since $\frac{-d_i}{2d_{i} + 1} = \frac{1}{\big( -2  - \frac{1}{d_i} \big)}$, $K$ is the boundary of the surface determined by the plumbing graph depicted in Figure \ref{fig: case(a)}(a). The latter can be modified using the plumbing calculus (\cite{Neu}) to show that $K$ is the boundary of the surface $S$ determined by the plumbing graph depicted in Figure \ref{fig: case(a)}(b). It follows that $S$ is obtained from the pretzel surface $F(-3, \cdots , -3, 1)$ (with $r + 1$ strands) by plumbing on each of the $-3$-twisted bands a bamboo of $d_{i} - 1$ positive Hopf bands, $1 \leq i \leq r$.  

\begin{figure}[!ht] 
\centering
 \includegraphics[scale=1]{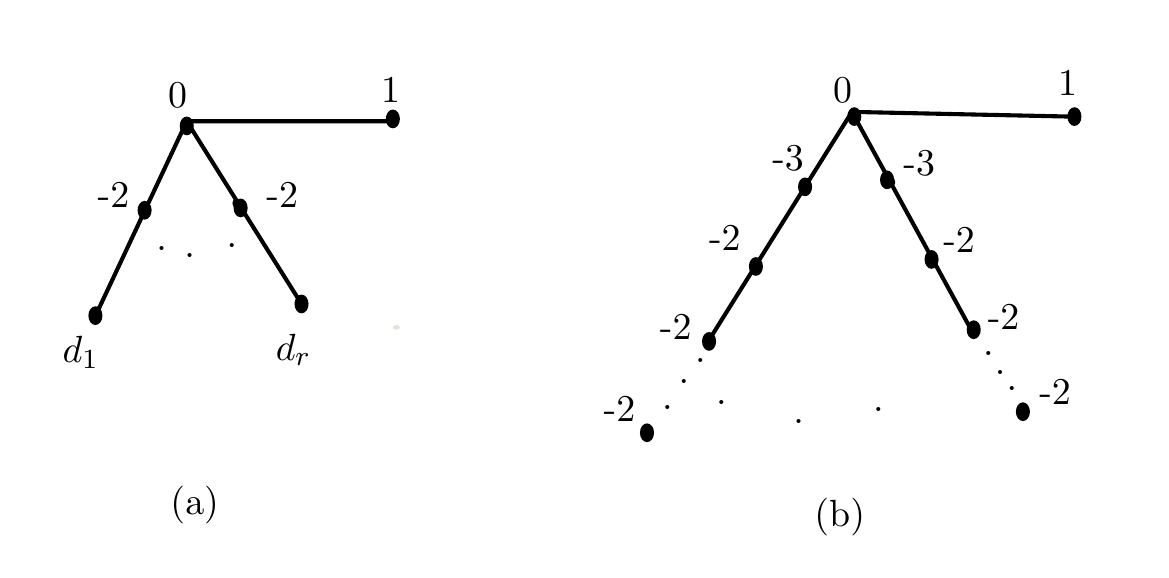}
 \caption{Plumbing diagrams for $K(\frac{-d_1}{2d_{1} + 1}, \cdots , \frac{-d_r}{2d_{r} + 1} \vert 1)$ } 
 \label{fig: case(a)}
\end{figure} 

The pretzel link $P(-3, \cdots , -3, 1)$ is fibred with fibre  $F(-3, \cdots , -3, 1)$ by \cite[page 538]{Ga2}.  As $S$ is obtained from the plumbing of a fibred surface with positive Hopf bands, $K$ is fibred with fibre $S$. In particular, $g(K) = g(S)$. 
  
Suppose that $|\sigma(K)| = 2g(K) = 2g(S)$. Then the symmetrised Seifert form of the Seifert surface $S$ of $K$ is definite. Since $S$ is obtained by plumbing Hopf bands to $F(-3, \cdots , -3, 1)$, the inclusion map $F(-3, -3 , -3, 1)\to S$ induces an injection on homology. Hence the symmetrised Seifert form of $F(-3, -3, -3, 1)$ is definite. But it is simple to see that the determinant of this form, i.e. the determinant of the link $P(-3,-3,-3,1)$, is zero, a contradiction. This completes the proof in case (a). 
 
\subparagraph{{\bf Case (b)}} 
Since $\frac{-m_i}{m_{i} + 1} = \frac{1}{\big( -1  - \frac{1}{m_i} \big)}$, $K$ is the boundary of the surface determined by the plumbing graph depicted in Figure \ref{fig: case(b)}(a), and therefore, via the plumbing calculus, it is the boundary of the fibre surface obtained by plumbing positive Hopf bands according to a star shaped tree depicted in Figure \ref{fig: case(b)}(c). 

\begin{figure}[!ht] 
\centering
 \includegraphics[scale=1]{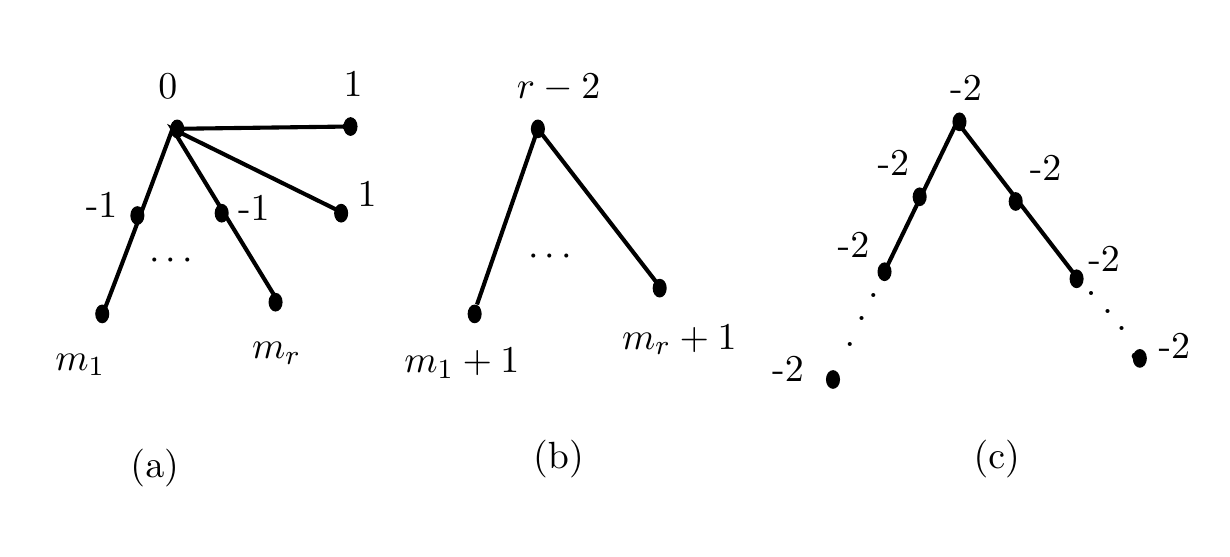}
 \caption{Plumbing diagrams for $K(\frac{-m_1}{m_{1} + 1}, \cdots , \frac{-m_r}{m_{r} + 1} \vert 2)$ } 
  \label{fig: case(b)}
\end{figure} 

Hence the fibre surface corresponds to a 
plumbing tree where each edge has weight $-2$. The only such trees whose associated quadratic forms are negative definite correspond to the simply laced Dynkin diagrams $A_n$, $D_n$, $E_6$, $E_7$ and $E_8$. See for example \cite[pages 61-62]{HNK}. The first case yields the torus knot or link $T(2, m+1)$ while the second and fourth yield links with two components; see \S \ref{sec: introduction}. Thus the only relevant cases for a Montesinos knot $K$ with three or more branches are $E_6$ and $E_8$, which correspond respectively to the $(3,4)$ and $(3,5)$ torus knot.
\end{proof}

\begin{remark}
{\rm Suppose that a quasipositive Seifert surface with connected boundary $K$ is given by the tree plumbing of twisted bands. Then \cite[Theorem 2.45]{Ru7} implies that every weight is even and strictly negative. If we require the surface to be a fibre of an open book, then the bands have to be positive Hopf bands (\cite[Theorem 3]{Ga1}). 
Conversely, if the boundary of a tree plumbing of positive Hopf bands is connected, it is a fibred strongly quasipositive arborescent knot. In this case the proof of case (b) of Proposition \ref{prop: fsqpm signature Michel} shows that $|\sigma(K)| < 2g(K)$ unless $K$ is a $(2,k)$, $(3,4)$, or $(3,5)$ torus knot. }
\end{remark}

\begin{prop} \label{prop: arborescent}
Let $K$ be an arborescent knot which bounds a Seifert surface obtained as a plumbing of positive Hopf bands along a tree. Then $K$ is definite if and only if it is a $(2,k)$, $(3,4)$, or $(3,5)$ torus knot. In particular, some $\Sigma_n(K)$ is an L-space if and only if $K$ is either a $(2,k)$ torus knot, the $(3,4)$ torus knot, or the $(3,5)$ torus knot. 
\qed
\end{prop}

We can now give the proof of  Corollary \ref{cor: positive braid or divide knot}.

\begin{proof}[Proof of Corollary \ref{cor: positive braid or divide knot}] Since $K$ strongly quasipositive,  if $\Sigma_n(K)$ is an L-space for some $n \geq 2$, then $|\sigma(K)| = 2g(K)$ by Proposition \ref{prop: main prop}. If $K$ is a strongly quasipositive prime alternating knot it must be a $(2,k)$-torus knot by Proposition \ref{prop: definite alternating}. If $K$ is a fibred Montesinos knot Proposition \ref{prop: fsqpm signature Michel} shows that  it must be a torus knot of type $(2,k)$, $(3,4)$ or $(3, 5)$. The same holds if $K$ is a positive braid knot by \cite{Baa2}, and if it is a divide knot by \cite[Theorem C]{BD}. (Beware that the link corresponding to the divide $Q$ is incorrectly depicted in Figure 0.1 of  \cite{BD}.)
\end{proof}

For the strongly quasipositive fibred knots considered in Corollary \ref{cor: positive braid or divide knot}, the proof shows that they are definite if and only if they belong to the following list: the $(2,k)$,
$(3,4)$ and $(3,5)$ torus knots. It might seem plausible that these are the only prime, fibered, definite, strongly quasipositive knots, but this is not the case. For instance there are genus $2$ definite plumbings of Hopf bands which are strongly quasipositive, fibred and hyperbolic knots, and thus do not belong to the above list. See \cite{Mis}. 

\section{Branched covers of non-fibred strongly quasipositive  knots} \label{sec: nfsqpk}

Theorem \ref{thm: sqp links} shows that once we know the Alexander polynomial and signature of a strongly quasipositive knot $K$, we can estimate at which value of $n$ the manifolds 
$\Sigma_n(K)$ become non-L-spaces. In this section we illustrate this point in several cases. 

\subsection{Strongly quasipositive knots of genus $1$} \label{sec: sqp genus 1}

The Alexander polynomial of a knot $K$ of genus $1$ is of the form
$$\Delta_K(t) = at^2 + (1-2a)t + a$$ 
where $a \in \mathbb Z$. When $a \leq 0$, $\Delta_K(t)$ has no roots on the unit circle so $\sigma_K \equiv 0$. When $a > 0$, the roots of $\Delta_K(t)$ are $\exp(\pm i \theta_a) \in S^1$ where $\theta_a = \arccos(1 - \frac{1}{2a})$. Further, 
$$|\sigma_K(\zeta)| = \left\{ 
\begin{array}{ll} 
0 & \hbox{if }  \zeta \in I_+(\exp(i \theta_a))   \\ 
1 & \hbox{if } \zeta = \exp(\pm i \theta_a) \\   
2 & \hbox{if } \zeta \in I_-(\exp(i \theta_a))   
\end{array} \right.$$
Consequently, Theorem \ref{thm: sqp links} implies

\begin{prop} 
\label{prop: genus 1 sqp} 
Suppose that $K$ is a strongly quasipositive knot of genus $1$ and Alexander polynomial $\Delta_K(t) = at^2 + (1-2a)t + a$. If some $\Sigma_n(K)$ is an L-space, then $a > 0$ and  $n < 2\pi/\textup{arccos}(1 - 1/2a)$. 
\qed
\end{prop}
 
Lee Rudolph has shown that any Alexander polynomial can be realized by a strongly quasipositive knot. See \cite[Theorem, \S3]{Ru2} or \cite[Proposition 87]{Ru7}. Therefore Proposition \ref{prop: genus 1 sqp} shows the existence of infinitely many  strongly quasipositive genus $1$ knots for which no $\Sigma_n(K)$ is an L-space.

\subsection{Alternating knots of genus $1$} 
The $2$-fold branched covers of alternating knots are L-spaces (\cite[Proposition 3.3]{OS2}) and so a necessary condition for such a knot $K$ to be strongly quasipositive is that $|\sigma(K)| = 2g(K)$. An alternating knot of genus $1$ is either a $2$-bridge knot or a $3$-strand pretzel knot $P(p,q,r)$ where $p,q,r$ are odd and of the same sign (\cite[Lemma 3.1]{BZ}).  We consider these two families separately. 

\subsubsection{Alternating pretzel knots of genus $1$} 

Consider the $3$-strand pretzel knot $K = P(p,q,r)$ where $p, q, r$ are odd and positive. Then $K$ is of genus $1$ and has Alexander polynomial $at^2  + (1-2a)t + a$ where $a = \frac{1 + pq + qr + rp}{4} > 0$. Rudolph has shown that $K$ is strongly quasipositive \cite[Lemma 3]{Ru3}. Thus from Proposition \ref{prop: genus 1 sqp}  we deduce, 

\begin{cor} \label{cor: alt pretzel genus 1} 
Let $K = P(p,q,r)$ where $p, q, r$ are odd and positive. Then $\Sigma_n(K)$ is not an L-space for $n \geq 2\pi/\textup{arccos}(1 - 2/(1 + pq + qr + rp))$. 
\end{cor} 

This corollary should be contrasted with the fact that for $n = 2,3, \Sigma_n(K)$ is an L-space. For $n=2$ this follows from \cite[Proposition 3.3]{OS2}, and for $n=3$ from \cite[Theorem 1.1]{Te3}.

\subsubsection{$2$-bridge knots of genus $1$} \label{sec: $2$-bridge genus 1} 

Up to taking mirror image, every $2$-bridge knot $K$ corresponds to a rational number $p/q$ where $p$ and $q$ are coprime integers such that $1 < q < p$, $p$ is odd and $q$ is even. Then $p/q$ has a continued fraction expansion with all terms even, and the number of terms in this expansion is $2g(K)$. It follows from Corollary \ref{cor:sqp alternating}  that $K$ is strongly quasipositive if and only if $|\sigma(K)| = 2g(K)$, see also \cite[Proposition 3.6 and Corollary 3.7]{BR}.  

Consider a $2$-bridge knot $K$ of genus 1, so $p/q = 2k - 1/2l = [2k,-2l]$ where $k$ and $l$ are integers with $k > 0$ and $l \ne 0$. The Alexander polynomial of $K$ is $\Delta_{K}(t) = klt^2 + (1 - 2kl)t + kl$. Since $|\sigma(K)| = 2$ if and only if $l > 0$, $K$ is strongly quasipositive if and only if $l > 0$ by Corollary \ref{cor:sqp alternating}. Here is a result which is a consequence of \cite{Pe} and Proposition \ref{prop: genus 1 sqp}.  

\begin{cor} \label{cor: 2bridge} 
Let $K$ be the $2$-bridge knot corresponding to the rational number $p/q = [2k,-2l]$ where $k > 0$ and $l \ne 0$.   

$(1)$ If $K$ is not strongly quasipositive $($i.e. $l < 0$$)$, each $\Sigma_n(K)$ is an L-space. 

$(2)$ If $K$ is strongly quasipositive $($i.e. $l > 0$$)$, $\Sigma_n(K)$ is not an L-space for $n \geq 2\pi/\textup{arccos}(1 - 1/2kl)$. 
\end{cor}

\begin{proof}
Assertion (1) follows from the main result of \cite{Pe}  while assertion (2) is a consequence of Proposition \ref{prop: genus 1 sqp}. 
\end{proof}

\begin{remarks}  \label{rem: genus one 2-bridge} 
{\rm 
(1) As noted above, if $k = l = 1$ then $K$ is the left-handed trefoil and the conclusion of Proposition 8.4(2) holds for $n \ge 6$. Since $\pi_1(\Sigma_n(K))$ is finite for $n = 2,3,4$ and 5, Corollary \ref{cor: 2bridge}(2) is best possible. 

(2) If $k = 2$ and $l = 1$ then $p/q = 7/2$ and $K$ is the knot $5_2$. Corollary \ref{cor: 2bridge} shows that $\Sigma_n(K)$ is not an L-space if $n \ge 9$. The situation for $2 \le n \le 8$ is as follows: The branched cover $\Sigma_n(K)$ is an L-space for $n = 2$ (since $\Sigma_2(K)$ is a lens space), $n = 3$ (\cite{Pe}), and $n = 4$ (\cite{Te1}). Robert Lipshitz has shown by computer calculation that $\Sigma_n(K)$ is an L-space for $n = 5$, and is not an L-space for $n = 6, 7$, and $8$. (The fact that $\Sigma_5(K)$ is an L-space was also proved by Mitsunori Hori. See \cite{Te2}). }
\end{remarks} 

\subsection{Non-alternating pretzel knots of genus $1$} 

Consider a non-alternating $3$-strand pretzel knot $K = P(p,q,r)$ where $p, q, r$ are odd. Then $K$ is of genus $1$ and has Alexander polynomial $at^2  + (1-2a)t + a$ where $a = \frac{1 + pq + qr + rp}{4}$.  Up to replacing $K$ by its mirror image, we can suppose that $p, q > 0$. Since $K$ is not alternating we have $r < 0$  and also that it is not $2$-bridge, so $\min\{p, q, |r|\} \geq 3$. It is known that $\Sigma_2(K)$ is an L-space if and only if $\min\{p, q\} \leq -r$ (see Lemma \ref{lemma: pretzel l-space cover} or  the discussion in \S 3.1 of \cite{CK}). In other words, $\min\{p, q\} + r \leq 0$. On the other hand, by Corollary \ref{cor: 3-strand knots} $K$ is strongly quasipositive if and only if $\min\{p, q\} + r > 0$.  

\begin{cor} \label{prop: non-alt pretzel of genus 1} 
Let $K$ be a strongly quasipositive non-alternating $3$-strand pretzel knot $P(p,q,r)$ where $p, q, r$ are odd with $p$ and $q$ positive $($so $r$ negative and  $\min\{p, q, |r|\} \geq 3$ $)$. If some $\Sigma_n(K)$ is an L-space, then $pq + qr + rp > 0$.  Further, $n < 2\pi/\textup{arccos}(1 - 2/(1 + pq + qr + rp))$. 
\end{cor} 

\begin{proof}
If some $\Sigma_n(K)$ is an L-space, then $|\sigma(K)| = 2$ so the calculations in \S \ref{sec: sqp genus 1}  imply that $0 < a = (1 + pq + qr + rp)/4$. Since $pq + qr + rp$ is odd it must be positive, so $pq + qr + rp > 0$. The final claim of the corollary follows from Proposition \ref{prop: genus 1 sqp}.
\end{proof}

\begin{remark}
{\rm Suppose that $K$ is as in Corollary \ref{prop: non-alt pretzel of genus 1}. Since $\Sigma_2(K)$ is not an L-space (Lemma \ref{lemma: pretzel l-space cover}), we expect that no $\Sigma_n(K)$ is an L-space.}
\end{remark}

\begin{cor} \label{cor: non-alt pretzel of genus 1} 
Let $K$ be a $3$-strand pretzel $P(p,q,r)$ where $p, q, r$ are odd and $p, q > 0$. If $\Delta_K(t) = 1$, then no branched cover $\Sigma_n(K)$ is an L-space. 
\end{cor} 

\begin{proof}
Recall that $K$ has Alexander polynomial $at^2  + (1-2a)t + a$ where $a = (1 + pq + qr + rp)/4$, so our hypotheses imply that $pq + qr + rp = -1$. Hence $r < 0$. Further, since the degree of the Alexander polynomial of an alternating knot is twice its genus, $K$ must be non-alternating. Hence $\min\{p, q, |r|\} \geq 3$. 

Now $p + q > 0$ by assumption and the identity $-1 =  pq + qr + rp = q(p+r) + pr = p(q+r) + qr$ implies that both $p+r$ and $q+r$ are positive since $pr, qr \leq -9$.  Hence $K$ is strongly quasipositive. Applying Proposition \ref{prop: genus 1 sqp}, if some $\Sigma_n(K)$ is an L-space, then $-1 = pq + qr + rp > 0$, a contradiction. Thus no $\Sigma_n(K)$ is an L-space.
\end{proof}

\subsection{Some $2$-bridge knots of large genus} \label{sec: 2-bridge}

Let $K(k,m)$ be the $2$-bridge knot corresponding to the rational number $(2m(2k-1)+1)/2m$, where $k,m \ge 1$. Note that $(2m(2k-1)+1)/2m$ has even continued fraction expansion $[2k,-2,2,-2,...2,-2]$, where the number of terms is $2m$. Hence $m = g(K(k,m))$.

The left-orderability of $\pi_1(\Sigma_n(K(k,m)))$ has been investigated by Tran \cite{Tra}; see \S \ref{subsubsec: Tran}. See also \cite[Theorem 4.3]{Hu}. 

\begin{remarks}

{\rm (1) Since the signature of $K(k,m)$ is $2m = 2g(K(k,m))$, $K(k,m)$ is strongly quasipositive by Corollary \ref{cor:sqp alternating}, see also \cite{BR}, and all the roots of its Alexander polynomial lie on $S^1$ (Proposition \ref{proposition: maximal signature}).

(2) $K(1,m)$ is the $(2,2m+1)$-torus knot.
}
\end{remarks}

\begin{lemma} \label{lemma: gae}
The Alexander polynomial of $K(k,m)$ is $k - (2k-1)t + (2k-1)t^2 - ...-(2k-1)t^{2m-1} + kt^{2m}$.
\end{lemma}

\begin{proof}
Recall that if $qq' \equiv -1$ (mod $p$) then the $2$-bridge knot $K_{p/q'}$ is the mirror image of $K_{p/q}$. Therefore, up to mirror image, $K(k,m)$ corresponds to the rational number $(2m(2k-1)+1)/(2k-1)$. Since $(2k-1)$ is odd, we can use the formula in \cite[Lemma 11.1]{Min} (see also \cite{HS}) to compute the Alexander polynomial $\Delta(t)$ of $K(k,m)$. This states that $\Delta(t) = \sum_{j = 0}^{2m(2k-1)} (-t)^{\sum_{i=1}^{j} \epsilon_i}$, where $\epsilon_i = (-1)^{r_i}$ and 
\begin{equation*}
r_i = \lfloor i(2k-1)/(2m(2k-1)+1) \rfloor, \hspace{2 mm} 1 \le i \le 2m(2k-1). 
\end{equation*}
Setting 
\begin{equation*}
I_s = \{i : 2m(s-1) + 1 \le i \le 2ms \}, \hspace{2 mm}  1 \le s \le 2k-1, 
\end{equation*}
we see that 

$$\epsilon_i = \left\{ 
\begin{array}{ll}
+1 & \mbox{if $i \in I_s$ and $s$ odd}, \\
-1 &\mbox{if $i \in I_s$ and $s$ even.}
\end{array} \right.$$

Hence, writing $\sigma(t) = -t+t^2-t^3+...-t^{2m-1}$, we have that
\begin{eqnarray} 
\Delta(t) & = & (1+ \sigma(t) + t^{2m}) + \sigma(t) + (1 + \sigma(t) + t^{2m}) + \sigma(t) + ... + (1 + \sigma(t) + t^{2m}) \nonumber \\
& = & (2k-1)\sigma(t) + k(1 + t^{2m}) \nonumber 
\end{eqnarray} 

\end{proof}

Since the case $m = 1$ is included in the class of knots discussed in \S \ref{sec: $2$-bridge genus 1}, we assume from now on that $m \ge 2$.
\medskip

\begin{lemma} \label{lemma: roots 2} $\;$ 

$(1)$ If $m \ge 3$, then $\Delta(t)$ has a root $\exp(i\theta)$ with $2\pi/3 < \theta < \pi$.

$(2)$ If $m = 2$, then $\Delta(t)$ has a root $\exp(i\theta)$ with $\pi/2 < \theta < \pi$. 
\end{lemma}

\medskip

\begin{cor} \label{cor: br covers 2bridge 2} $\;$ 

$(1)$ If $m \ge 3$, then $\Sigma_n(K(k,m))$ is an L-space if and only if $n = 2$.

$(2)$ If $m = 2$, then $\Sigma_n(K(k,m))$ is an L-space if and only if $n = 2, 3$. 

\end{cor}

\begin{remarks} $\;$ 

{\rm (1) The (2,5)-torus knot shows that part (2) of Corollary \ref{cor: br covers 2bridge 2} is best possible.

(2) One can show that conclusion (1) of Lemma \ref{lemma: roots 2} fails for $K(k,2)$ for all $k \ge 1$.}

\end{remarks}

\begin{proof}[Proof of Lemma \ref{lemma: roots 2}]
By Lemma \ref{lemma: gae}, 
\begin{equation*}
\Delta(t) = k(1-t+t^2-...+t^{2m}) - (k-1)(t-t^2+t^3-...+t^{2m-1}).
\end{equation*}

Define $h(t) = (1+t)\Delta(t) = k(1+t^{2m+1}) - (k-1)t(1+t^{2m-1})$. We will show that $h(t)$ has a root of the form stated.

Symmetrizing, let $g(t) = t^{-(2m+1)/2}h(t) = k(t^{-(2m+1)/2} + t^{(2m+1)/2}) + (k-1)(t^{-(2m-1)/2} + t^{(2m-1)/2})$. 

Writing $t = \exp(i\theta)$, the roots of $g(t)$ correspond to the roots of 
\begin{equation*}
G(\theta) = k\cos((2m+1)\theta/2)  - (k-1)\cos((2m-1)\theta/2).
\end{equation*}
Note that $G(\pi) = 0$.

Taking derivatives, and using 
$$\sin((2m+1)\pi/2) = -\sin((2m-1)\pi/2) = \left\{ 
\begin{array}{ll}
+1 & \hbox{if $m$ even}, \\
-1 & \hbox{if $m$ odd},  
\end{array} \right.$$

we see that  
\begin{equation} \label{G'}
G'(\pi) \hbox{ is } \left\{ 
\begin{array}{ll} 
< 0 & \hbox{if $m$ even}, \\
> 0 & \hbox{if $m$ odd}.
\end{array} \right.
\end{equation}
Let $\theta_0 = ((m-1)/(2m+1))2\pi$.
Then $G(\theta_0) = k\cos((m-1)\pi) - (k-1)\cos \varphi$, for some $\varphi$. Since 
$$\cos((m-1)\pi) = \left\{ 
\begin{array}{ll}
-1 & \hbox{if $m$ even}, \\
+1 & \hbox{if $m$ odd}, 
\end{array} \right.$$ 
we have
\begin{equation} \label{G}
G(\theta_0) \hbox{ is } \left\{ 
\begin{array}{ll} 
< 0 & \hbox{if $m$ even}, \\
> 0 & \hbox{if $m$ odd}.
\end{array} \right.
\end{equation}
Combining (\ref{G'}) and (\ref{G}) we see that $G(\theta)$ has a root $\in (\theta_0, \pi)$. 

If $m\ge 4$ then $(m-1)/(2m+1) \ge 1/3$, and so $\theta_0 \ge 2\pi/3$.

If $ m = 3$, $G(2\pi/3) = k\cos(7\pi/3) - (k-1)\cos(5\pi/3) = \cos(\pi/3) > 0$. 
So $G(\theta)$ has a root $\in (2\pi/3, \pi)$.

Finally, if $m = 2$, $G(\pi/2) = k\cos(5\pi/4) - (k-1)\cos(3\pi/4) = \cos(3\pi/4) < 0$. Hence $G(\theta)$ has a root in $(\pi/2, \pi)$.
\end{proof}

Since $\Sigma_2(K(k, m))$ is an L-space for all $k, m \geq 1$, the only part of Corollary \ref{cor: br covers 2bridge 2} not covered by Lemma \ref{lemma: roots 2} and Theorem \ref{thm: sqp links}(1) is showing that $\Sigma_3(K(k, 2))$ is an L-space for all $k \geq 1$. We do this now. 

Recall that $K(k, 2)$ is the $2$-bridge knot corresponding to the rational number with continued fraction $[2k, -2, 2, -2]$, $k \geq 1$. Let $T(k)$ be the tangle shown in Figure \ref{fig: cam 16}, where the box labeled $k$ denotes $k$ vertical negative half-twists.

\begin{figure}[!ht] 
\centering
 \includegraphics[scale=.7]{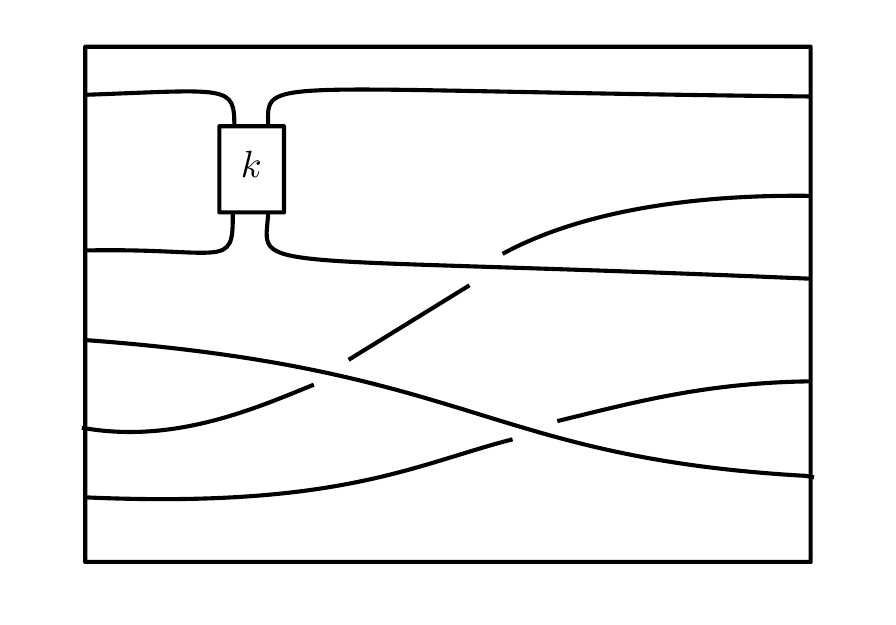}
\caption{The tangle $T(k)$} 
\label{fig: cam 16}
\end{figure} 

Let $L(k,k,k)$ be the (unoriented) link obtained by concatenating three copies of $T(k)$ and identifying the two ends of the resulting tangle in the obvious way. By \cite[Theorem 8]{MV}, $\Sigma_3(K(k, 2)) \cong \Sigma_2(L(k,k,k))$ (cf. \cite[Proposition 1]{Te1}). We will show that $\Sigma_2(L(k,k,k))$ is an L-space. 

It is actually more convenient to prove a more general result. To this end, let $L(k_1,k_2,k_3)$ be the link obtained by connecting $T(k_1), T(k_2)$, and $T(k_3)$ and connect the ends as before. Here we allow $k_i$ to also take the value $\infty$, meaning that we substitute the tangle in Figure \ref{fig: cam 17}. 

\begin{figure}[!ht]
\centering
 \includegraphics[scale=0.7]{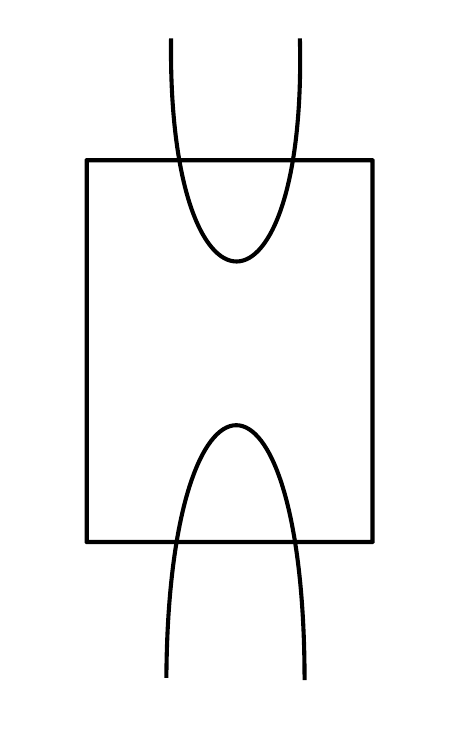}
\caption{} 
\label{fig: cam 17}
\end{figure}  

Clearly, $L(k_1,k_2,k_3)$ is unchanged by a cyclic permutation of $(k_1,k_2,k_3)$.

\begin{thm} \label{thm: Sigma_2(L(k_1,k_2,k_3)) is an L-space}
Suppose that $k_1, k_2, k_3$ are positive integers. Then $\Sigma_2(L(k_1,k_2,k_3))$ is an L-space. 
\end{thm}

\begin{cor}
$\Sigma_3(K(k, 2))$ is an L-space. 
\qed
\end{cor}

This completes the proof of Corollary \ref{cor: br covers 2bridge 2} modulo that of Theorem \ref{thm: Sigma_2(L(k_1,k_2,k_3)) is an L-space}, which we deal with now. 

\begin{lemma} \label{lemma: simple parameters} $\;$

$(1)$ $L(k,\infty,\infty) = T(3,4)$ for all $k \geq 1$.

$(2)$ $L(1,1, \infty) = P(-2,3,4)$. 

$(3)$ $L(1,1,1) = T(3,5)$.

\end{lemma} 

\begin{proof}
(1) It is clear that $L(k,\infty,\infty)$ is independent of $k$, and is the closure of $\beta^3$, where $\beta$ is the $4$-braid $\sigma_3 \sigma_1 \sigma_2$. With ``$\sim$" denoting conjugacy, $\sigma_3 \sigma_1 \sigma_2 \sim \sigma_1 \sigma_2\sigma_3$. Thus $\beta^3 \sim ( \sigma_1 \sigma_2\sigma_3)^3$, whoose closure is $T(3,4)$.

(2) This can be shown by direct diagram manipulation.

(3) $L(1,1,1)$ is the closure of $\beta^3$, where $\beta$ is the $5$-braid $\sigma_2 \sigma_4 \sigma_1 \sigma_3 = \sigma_2 \sigma_1 \sigma_4 \sigma_3 \sim \sigma_4 \sigma_3 \sigma_2 \sigma_1$. Hence $\beta^3 \sim (\sigma_4 \sigma_3 \sigma_2 \sigma_1)^3$, whose closure is $T(3,5)$.
\end{proof}

Let $L_, L_0, L_\infty$ be links that differ only locally as shown in Figure \ref{fig: cam 18}. 

\begin{figure}[!ht] 
\centering
 \includegraphics[scale=0.9]{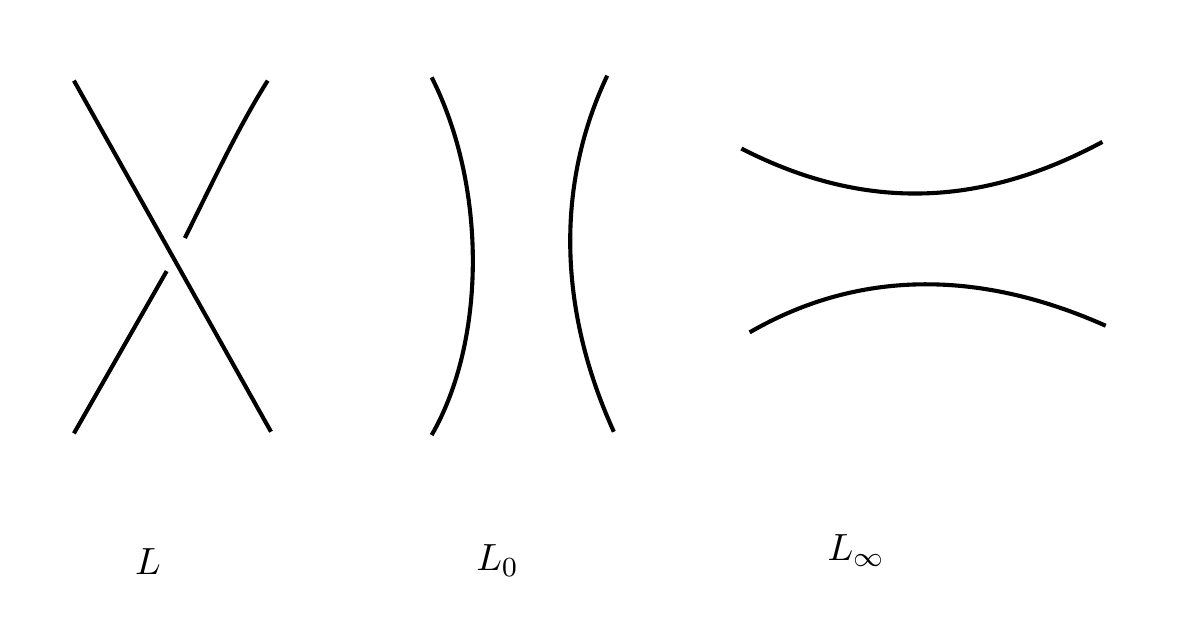}
\caption{} 
\label{fig: cam 18}
\end{figure}  

We shall say that $(L, L_0, L_\infty)$ is an {\it additive triple} if $\det L_0 \ne 0, \; \det L_\infty \ne 0$ and $\det L = \det L_0 + \det L_\infty$. In the proof of Theorem \ref{thm: Sigma_2(L(k_1,k_2,k_3)) is an L-space} we will make repeated use of the fact that if $(L, L_0, L_\infty)$ is an additive triple, and if $\Sigma_2(L_0)$ and $\Sigma_2(L_\infty)$ are L-spaces, then $\Sigma_2(L)$ is an L-space. See \cite{OS2}.

\begin{lemma} \label{lemma: additive dets}
Suppose that $k_1, k_2, k_3$ are positive integers.

$(1)$ $\det L(k_1,k_2,k_3) = 3(k_1k_2 + k_2k_3 + k_3k_1)  - 4(k_1 + k_2 + k_3) + 4$. 

$(2)$ $\det L(k_1,k_2, \infty) = 3(k_1 + k_2) - 4$.
\end{lemma}

\begin{proof}
These formulae are obtained by computing the determinants of the Goeritz matrices associated with suitable checkerboard shadings of the obvious diagrams of $L(k_1,k_2,k_3)$ and $L(k_1,k_2, \infty)$. 

Note that under our assumptions on $k_1, k_2, k_3$ the expressions on the right-hand sides of (1) and (2) are positive.
\end{proof}

\begin{cor} \label{cor: additive triples}
Suppose that $k_1, k_2, k_3$ are positive integers. The following are additive triples.

$(1)$ $(L(k_1,k_2 + 1, \infty), L(k_1,k_2,\infty), T(3, 4))$.

$(2)$ $(L(k_1, k_2, k_3 + 1), L(k_1,k_2, k_3), L(k_1,k_2,\infty))$.

\end{cor}

\begin{proof}
It is clear that for each triple the links differ locally as in Figure \ref{fig: cam 18}, recalling that in (1) $L(k_1,\infty,\infty) = T(3,4)$ by Lemma \ref{lemma: simple parameters}(1). 

(1) Lemma \ref{lemma: additive dets}(2) implies that $\det L(k_1,k_2 + 1, \infty) = \det L(k_1,k_2,\infty) +3$, and since $\det T(3, 4) = 3$, the result follows.

(2) Lemma \ref{lemma: additive dets}(1) and (2) give $\det L(k_1,k_2,k_3+1) = \det L(k_1,k_2,k_3) + \det L(k_1,k_2, \infty)$. Hence the stated triple is additive. 
\end{proof}

\begin{lemma} \label{Sigma2 k1k2inf}
Suppose that $k_1, k_2$ are positive integers. Then $\Sigma_2(L(k_1, k_2, \infty))$ is an L-space.
\end{lemma}

\begin{proof}
We induct on $k_1 + k_2$, the base case being $L(1,1,\infty) = P(-2, 3, 4)$ by Lemma \ref{lemma: simple parameters}(2). 

Suppose that $k_1 + k_2 > 2$. Assume $k_2 > 1$; the case $k_1 > 1$ is similar, using the appropriate analogue of Corollary \ref{cor: additive triples}(1). By our inductive hypothesis, 
$\Sigma_2(L(k_1, k_2 - 1, \infty))$ is an L-space. Since $\Sigma_2(T(3,4))$ is also an L-space, Corollary \ref{cor: additive triples}(1) implies that $\Sigma_2(L(k_1, k_2, \infty))$ is an L-space. 
\end{proof}

\begin{proof}[Proof of Theorem \ref{thm: Sigma_2(L(k_1,k_2,k_3)) is an L-space}]
We induct on $k_1 + k_2 + k_3$ where the base case is $L(1,1,1) = T(3,5)$ by Lemma \ref{lemma: simple parameters}(3). Suppose $k_1 + k_2 + k_3 > 3$. By cyclically permuting $(k_1, k_2, k_3)$ we may assume that $k_3 > 1$. Since $\Sigma_2(L(k_1, k_2, \infty))$ is an L-space by Lemma \ref{Sigma2 k1k2inf}, the result follows by induction from Corollary \ref{cor: additive triples}(2)
\end{proof}

\section{Extensions to quasipositive links}  \label{sec: qp links}

Throughout this section we suppose that $L$ is a quasipositive link and that up to replacing $L$ by its mirror image, it is the boundary of a properly embedded surface $F \subset B^4$ which is the intersection of a complex affine curve in $\mathbb{C}^2$ with $B^4$. Then $G(F) = G_4(L)$ (\cite[Proposition, \S 3]{Ru3}) and $\Sigma_n(F)$ is a Stein filling of $\Sigma_n(L)$ (\cite[Theorem 1.3]{HKP}, \cite[Theorem 1.3]{Ru7}).

\begin{proof}[Proof of Theorem \ref{thm: qp links}]
Suppose that $\Sigma_n(L)$ is an L-space for some $n \geq 2$ and let $\mu = |F|$. 

Since $\Sigma_n(L)$ is a rational homology $3$-sphere, $\eta_L(\zeta_n^j) = 0$ for $1 \leq j \leq n-1$ and $\Delta_L(\zeta_n^j) \ne 0$ for $1 \leq j \leq n-1$. 
Further, the intersection form on $H_2(\Sigma_n(F))$ is non-singular and the fact that $\Sigma_n(L)$ is an L-space implies that  $\beta_2^+(\Sigma_n(F)) = 0$. Then for $1 \leq j \leq n-1$, 
\begin{eqnarray} 
2g(F) + (m-\mu) + \beta_1(\Sigma_n(F);j) + \beta_3(\Sigma_n(F); j) & = & \beta_2(\Sigma_n(F); j) \hbox{ by Lemma \ref{lemma: chi}} \nonumber \\
& = &  |\sigma_L(\zeta_n^j)|  \nonumber \\ 
& \leq & 2g(F) + (m-\mu) - (\mu -1) \hbox{ by Corollary \ref{cor: florens inequality qhs}} \nonumber 
\end{eqnarray} 
It follows that $\mu = 1$ and $\beta_1(\Sigma_n(F);j) = \beta_3(\Sigma_n(F); j) = 0$. Hence 
$$|\sigma_L(\zeta_n^j)|  = 2g(F) + (m-1) \hbox{ for } 1 \leq j \leq n-1,$$ 
and since $G_4(L) = G(F) = g(F) + (m-1)$, we have $|\sigma_L(\zeta_n^j)| + (m-1) = 2G_4(L)$. Proposition \ref{prop: florens inequality qhs 2} now shows that $G_4(L) = G_4^{top}(L)$. 

Suppose that $F'$ is a locally flat surface of $\mu'$ components properly embedded in $B^4$ which realises $g_4^{top}(L)$. Then $g(F') + (m-\mu') = G(F') \geq G_4^{top}(L) = G(F) = g(F) + (m-1)$. Then $g_4^{top}(L) = g(F') \geq g(F) + (\mu'-1) \geq g(F) \geq g_4^{top}(L)$. Hence $\mu' = 1$ and $g(F) = g_4^{top}(L)$.  It follows that $g_4(L) = g_4^{top}(L)$. Then 
$$|\sigma_L(\zeta_n^j)|  = 2g_4(L) + (m-1)  = 2g_4^{top}(L) + (m-1) \hbox{ for }1 \leq j \leq n-1.$$ 
Part (1) of the theorem follows from these observations. 

Suppose that $L'$ is a link of $m$ components for which $\Delta_{L'}(t)$ is not divisible by $(t-1)^{2g_4(L') + (m-1)}$ and for which there is some $r \geq 2$ such that for all 
$j \in \{1, 2, \ldots , r-1\}$ we have $|\sigma_{L'}(\zeta_r^j)| = 2g_4(L') + (m-1)$. Define
$$n_4(L') = \max\{n \geq 2: \displaystyle \sum_{\zeta \in I_+(\zeta_n)} Z_\zeta(\Delta_{L'}(t)) \geq 2g_4(L') + (m-1)\},$$
which is well-defined by (\ref{matumoto}) and the remarks following it. It is clear that $n_4(L)$ depends only on $\Delta_{L}(t)$ and $g_4(L)$ and, from part (1), that $\Sigma_k(L)$ is not an L-space for $k > n_4(L)$, which is (2).

Finally, assuming the hypothesis of (3), its conclusion follows from the fact that $|\sigma_L(\zeta_n^j)| = \beta_2(\Sigma_n(F)) = 2g(F) + (m-1) = 2g(L) + (m-1)$ for $1 \leq j \leq n-1$ and we now proceed as in the proof of Theorem \ref{thm: sqp links}. 
\end{proof}

\begin{examples} \label{examples: qp}
{\rm Theorem \ref{thm: qp links} has some interesting consequences for knots $K$ with $12$ or fewer crossings.

(1) KnotInfo lists 296 knots of 12 or fewer crossings as being quasipositive. Of these, 156 are alternating and a further 78 have been identified as quasi-alternating (\cite{Jbn}). Two more are H-thin (\cite{Jbn}) and Nathan Dunfield has calculated that one other, $11n126$, has an L-space $2$-fold branched cover (private communication). In all then, 237 of the 296 listed quasipositive knots of 12 or fewer crossings have L-space $2$-fold branched covers. Corollary \ref{cor: qa and qp links} then calculates that $g_4^{top}(K) = g_4(K) = \frac12 |\sigma(K)|$ for these knots, which can be verified in KnotInfo. The remaining 59 knots are, interestingly, strongly quasipositive. Thus Corollary \ref{cor: qa and qp links} calculates the smooth and locally flat $4$-ball genera of all listed quasipositive, non-strongly quasipositive knots of 12 or fewer crossings. Further, of the 59 strongly quasipositive knots whose $2$-fold branched covers do not appear to be L-spaces, 36 have different smooth and locally flat $4$-ball genera. For such knots, no $\Sigma_n(K)$ can be an L-space. 

(2) Of the forty-five knots $K$ with $12$ or fewer crossings that KnotInfo lists as quasipositive but not strongly quasipositive, six are slice. For the remaining thirty-nine we can obtain an upper bound for $n_4(K)$ by examining their signature functions. For instance, the signature function of $K = 12n_{-}0234$ is non-positive decreasing and attains its maximum absolute value at  $0.678089 \pi$ (cf. KnotInfo). Hence $n_4(K) = \lfloor \frac{2}{0.678089} \rfloor = 2$ and so $\Sigma_n(K)$ is not an L-space for $n \geq 3$ by Theorem \ref{thm: qp links}. Applying Theorem \ref{thm: qp links} to the remaining thirty-eight knots we have that $\Sigma_n(K)$ is not an L-space for $n$ at least
\vspace{-.2cm} 
\begin{itemize}

\item $4$ when $K$ is either $10_{-}127, 12n_{-}0113, 12n_{-}0114, 12n_{-}0191, 12n_{-}0233, 12n_{-}0344, \\ 12n_{-}0466, 12n_{-}0570, 12n_{-}0674, 12n_{-}0684, 12n_{-}0707, 12n_{-}0722, 12n_{-}0747, 12n_{-}0820, \\ 12n_{-}0882$, or $12n_{-}0887$;  

\vspace{.2cm} \item 5 when $K$ is either $10_{-}149, 10_{-}157, 12n_{-}0190, 12n_{-}0683$, or $12n_{-}0831$; 

\vspace{.2cm} \item $7$ when $K$ is either $8_{-}21, 10_{-}143, 10_{-}159, 12n_{-}0604, 12n_{-}0666$, or $12n_{-}0767$; 

\vspace{.2cm} \item $8$ when $K$ is either $10_{-}131, 10_{-}133, 12n_{-}0467$, or $12n_{-}0822$; 

\vspace{.2cm} \item $9$ when $K$ is either $12n_{-}0345$ or $12n_{-}0748$; 

\vspace{.2cm} \item $10$ when $K$ is $9_{-}45$; 

\vspace{.2cm} \item $11$ when $K$ is either $10_{-}148, 10_{-}165$ or $12n_{-}0829$; 

\vspace{.2cm} \item $13$ when $K$ is $10_{-}126$.

\end{itemize}
}
\end{examples}

\begin{question}
{\rm Are there quasipositive non-strongly quasipositive knots with $g_4^{top}(L) < g_4(L)$? For such $L$, no $\Sigma_n(L)$ can be an L-space. }
\end{question}

\begin{cor} \label{cor: trivial alexander} Let $K$ be a quasipositive knot with trivial Alexander polynomial. If $\Sigma_n(K)$ is an L-space for some $n \geq 2$, then $K$ is a slice knot. In particular, if $K$ is knotted, it cannot be strongly quasipositive. 
\end{cor}

\begin{proof} By Freedman (\cite[Theorem 1.13]{Fr}), $g_4^{top}(K) = 0$ for a knot $K$ with trivial Alexander polynomial. If $\Sigma_n(K)$ is an L-space for some $n \geq 2$, $g_4(K) = g_4^{top}(K) = 0$ by Theorem \ref{thm: qp links}, and so $K$ is slice. The second assertion follows immediately from the fact that $g_4(K) = g(K)$ for a strongly quasipositive knot.
\end{proof}

Here is an application to quasipositive satellite knots. 

\begin{cor} \label{cor: satellite} 
Suppose that $K = P(C)$ is a satellite knot with non-trivial companion $C$ and pattern $P$ of winding number $w$. Let $K_1 = P(U)$ where $U$ is the unknot. Suppose that $K$ and $C$ are quasipositive and $P$ is contained in its unknotted solid torus as the closure of a quasipositive $|w|$-braid where $|w| \geq 1$. 
If $\Sigma_n(K)$ is an L-space for some $n \geq 2$, then $|\sigma_{C}(\zeta_n^{wj})| = 2g_4(C)$ and $|\sigma_{K_1}(\zeta_n^j)| = 2g_4(K_1)$ for $1 \leq j \leq n-1$. Further, $|w| = 1$ if $C$ is not smoothly slice.
\end{cor}

\begin{proof} 
By hypothesis $K, C$, and $P$ are quasipositive with $P$ contained in its unknotted solid torus as the closure of a quasipositive $|w|$-braid where $|w| \geq 1$. According to  J\"{o}ricke \cite[Lemma 1]{Jo}, $g_4(K) = |w|g_4(C) + g_4(K_1)$. If $\Sigma_n(K)$ is an L-space for some $n \geq 2$,  Theorem \ref{thm: qp links} implies that $|\sigma_K(\zeta_n^j| = 2g_4(K)$ for $1 \leq j \leq n-1$. By using Litherland 's result (\cite[Theorem 2]{Lith}) and the same argument as in the proof of Proposition \ref{prop: satellite},  one concludes that $|\sigma_{C}(\zeta_n^{wj})| = 2g_4(C)$, $|\sigma_{K_1}(\zeta_n^j)| = 2g_4(K_1)$, for $1 \leq j \leq n-1$, and that $|w| = 1$ if $g_4(C) > 0$. 
\end{proof}

\section{Questions, problems, and remarks} \label{sec: questions and problems}

\subsection{The expected form of $\mathcal{L}_{br}(L)$} 

The expected form of $\mathcal{L}_{br}(L)$ (cf. \S \ref{sec: introduction}) suggests the following problem. 

\begin{problem} \label{prob: expected}
{\rm Show that if $L$ is a strongly quasipositive knot such that $\Sigma_n(L)$ is an L-space, then $\Sigma_r(K)$ is an L-space for each $2 \leq r \leq n$. }
\end{problem}

Let $L$ be as in Problem \ref{prob: expected} and choose a connected Seifert surface of $L$ which can be isotoped, relative to $L$, to a properly embedded surface $F \subset B^4$ which is the intersection of $B^4$ with a complex affine curve in $\mathbb C^2$. If Problem \ref{prob: expected} has a positive solution, then $\Sigma_r(L)$ would be a rational homology $3$-sphere and $\beta_2^+(\Sigma_r(F)) = 0$ would hold for $2 \leq r \leq n$. This is indeed the case. 

\begin{prop} \label{cor: b2+ = 0 filling}
Suppose that $L$ is a strongly quasipositive link such that $\Sigma_n(L)$ is an L-space for some $n \geq 2$. Let $F$ be as above. Then for each $2 \leq r \leq n$, $\Sigma_r(L)$ is a rational homology $3$-sphere  and $\beta_2^+(\Sigma_r(F)) = 0$. 
\end{prop}

\begin{proof} 
Fix $r \in \{2, 3, \ldots, n\}$. We already know that $\Sigma_r(F)$ is a simply-connected Stein filling of $\Sigma_r(L)$. We show below that $\Sigma_r(L)$ is a rational homology $3$-sphere and  $\beta_2^+(\Sigma_r(F)) = 0$. 

It follows from Proposition \ref{prop: main prop} that there are no roots of $\Delta_L(t)$ in the closed sector $\bar{I}_-(\zeta_n)$, and since this sector contains all $r^{th}$ roots of unity other than $1$, $\Sigma_r(L)$ is a rational homology $3$-sphere (cf. \S \ref{sec: signature function}). Proposition \ref{prop: main prop} also implies that the restriction of $\sigma_L$ to $I_-(\zeta_n)$ is constantly $-(2g(L) +(m-1))$, so in particular if $1 \leq j \leq r-1$, then $\sigma_L(\zeta_r^{j}) = -(2g(L) + (m-1)) = -2(g(F) + (m-1))= - \beta_2(\Sigma_r(F); j) \hbox{ (Lemma \ref{lemma: chi} and Remark \ref{rem: seifert surface case})} = -\hbox{dim}_{\mathbb C} (E_2(\Sigma_r(F));j)$. On the other hand, $\sigma_L(\zeta_r^j)$ equals $\hbox{signature}(\langle \cdot, \cdot\rangle_{\Sigma_r(F)}|E_2(\Sigma_r(F));j)$ (Theorem \ref{thm: viro}), and so $\langle \cdot, \cdot\rangle_{\Sigma_r(F)}|E_2(\Sigma_r(F);j)$ is negative definite. Hence as $H_2(\Sigma_r(F); \mathbb C)$ is isomorphic to the sum $\oplus_{j=1}^{r-1} E_2(\Sigma_r(F);j)$, $\langle \cdot, \cdot\rangle_{\Sigma_r(F)}$ is negative definite. Thus $\beta_2^+(\Sigma_r(F)) = 0$.
\end{proof}

\subsection{Extensions to more general branched covers} 

\subsubsection{Branched covers of manifolds obtained by surgery on knots in $S^3$} We will write $\mathcal{L}_{br}(K; \mu)$ for $\mathcal{L}_{br}(K)$ for reasons which will become  clear below. 

There are various ways to associate branched covers to the manifolds obtained by surgery on knots in the $3$-sphere. Here is one.

Let $K$ be a knot in the $3$-sphere with exterior $M$. For each $r = p/q \in \mathbb Q \setminus \{0\}$ in lowest terms we can associate a slope $\alpha \leftrightarrow \pm(p \mu_K + q \lambda_K) \in H_1(\partial M_K)/\pm$ where $M_k$ is the exterior of $K$ and $\mu_K$ and $\lambda_K$ are meridional and longitudinal classes of $K$. If $n \geq 2$ is an integer relatively prime to $p$, there is an $n$-fold cyclic cover $\Sigma_n(K; \alpha)$ of the Dehn filling $K(\alpha)$ branched over the core of the $\alpha$-filling torus. (It's easy to see that $\Sigma_n(K; \alpha)$ can be obtained by Dehn filling the $n$-fold cyclic cover of $M_K$.) Set 
$$\mathcal{L}_{br}(K, \alpha) = \{n \geq 2: \Sigma_n(K, \alpha) \hbox{ is an L-space }\}$$ 

\begin{question}
{\rm What are the constraints on $\mathcal{L}_{br}(K, \alpha)$? For instance, is it always a possibly empty set of successive integers begining with $2$? }
\end{question}
The sets $\mathcal{L}_{br}(K, \alpha)$ can be calculated in the case that $K$ is a torus knot with interesting results.    

\begin{question}
{\rm What is the geometry of 
$$\mathcal{L}_{br}(K; \partial M) = \{n \alpha \in H_1(\partial M) : \Sigma_n(K; \alpha)\hbox{ is an L-space}\}$$ 
considered as a subset of $H_1(\partial M; \mathbb Z)  $($\cong \mathbb Z^2$)$ \subset H_1(\partial M; \mathbb R) $($\cong \mathbb R^2$)?}
\end{question}

\subsubsection{Branched covers of manifolds other than $S^3$} 
The definition of (strongly) quasipositive links has been extended to links in arbitrary closed, connected, orientable $3$-manifolds and Hayden has shown that transverse $\mathbb C$-links in such manifolds are quasipositive (\cite{Ha}). Conversely, he has shown that quasipositive links contained in Stein-fillable contact $3$-manifolds bound symplectic surfaces in some Stein domain. Thus many of our arguments can be applied in a more general setting. 

\begin{question} 
{\rm In what form do our results extend to quasipositive links in rational homology $3$-spheres?}
\end{question}

\subsection{Branched covers of slice quasipositive knots} 
The constant $n_4(L)$ of Theorem \ref{thm: qp links} is defined only for quasipositive links for which $\Delta_L(t)$ is not divisible by $(t-1)^{2g_4(L) + (m-1)}$, and for such links $\Sigma_n(L)$ is not an L-space for $n > n_4(L)$. As it is not defined for smoothly slice quasipositive. knots, we are led to the following question.  
 
\begin{question} 
{\rm Is there a smoothly slice quasipositive knot all of whose branched covers are L-spaces?}
\end{question} 

\subsection{Branched cover integer homology spheres} 
Ozsv\'ath and Szab\'o have conjectured that the only irreducible integer homology $3$-sphere L-spaces are $S^3$ and the Poincar\'e homology $3$-sphere (\cite[Problem 11.4 and the remarks which follow it]{Sz}). Hence if $K$ is prime and $\Sigma_n(K)$ is an integer homology $3$-sphere L-space, the orbifold theorem leads us to expect $K$ to be either the $(2,3), (2, 5)$ or $(3, 5)$ torus knot.

\begin{problem}
{\rm Show that if $K$ is a prime quasipositive knot and some $\Sigma_n(K)$ is a $\mathbb Z$-homology $3$-sphere L-space, then $K$ is either the $(2,3), (2, 5)$ or $(3, 5)$ torus knot (so respectively, $n$ is $5, 3,$ or $2$).}
\end{problem}

Our methods do provide constraints on strongly quasipositive knots for which some $\Sigma_n(K)$ is a $\mathbb Z$-homology $3$-sphere L-space. For instance, if $n = 2$,  the genus of $K$ must be divisible by $4$ or equivalently, the signature of $K$ is a multiple of $8$. To see this, let $F$ be a quasipositive Seifert surface for $K$ and recall that there is an isomorphism between $H_1(F)$ and $H_2(\Sigma_2(F))$ under which the intersection form on $H_2(\Sigma_2(F))$ is isomorphic to the symmetrised Seifert form $\mathcal{S}_F$ (cf. \S \ref{subsec: signatures and branched covers}). The intersection form on $H_2(\Sigma_2(F))$ is unimodular since $\Sigma_2(F)$ is simply-connected and $\Sigma_2(K)$ is a $\mathbb Z$-homology $3$-sphere. It is also an even form, as this is clearly true of $\mathcal{S}_F$. Thus its signature is a multiple of $8$ (cf. \cite[Corollary 1, \S V.2]{Ser}), and since $K$ is a strongly quasipositive knot, this signature is twice the genus of $K$ (Theorem \ref{thm: sqp links}(1)), which completes the argument. 

The same argument implies that if the $\mu$-invariant of $\Sigma_2(K)$ is zero, then the genus of $K$ is divisible by $8$. 

\subsection{Quasipositive links with LO or CTF branched cyclic covers} 
It has been conjectured that for closed, connected, orientable, irreducible $3$-manifolds, the conditions of not being an L-space (NLS), of having a left-orderable fundamental group (LO), and of admitting a co-oriented taut foliation (CTF) are equivalent (\cite{BGW}, \cite{Ju}). Thus it is natural to ask whether the analogues of our results hold with the condition of some branched cyclic cover being an L-space replaced by either it not admitting a co-oriented taut foliation or it not having a left-orderable fundamental group. 

\begin{problem}
{\rm Prove analogous results to Theorem \ref{thm: sqp links} and Theorem \ref{thm: qp links} for strongly quasipositive links for which some $\Sigma_n(L)$ does not admit a co-oriented taut foliation or for which some $\Sigma_n(L)$  has a non-left-orderable fundamental group.}
\end{problem}

For instance, it is known that the conjectures hold for graph manifolds (\cite{BC}, \cite{HRRW}) and since the $2$-fold branched covers of arborescent links are of this type, we deduce: 

\begin{prop}

Let $L$ be an arborescent link and suppose that either $\Sigma_2(L)$ is not CTF or $\Sigma_2(L)$ is not LO. 

$(1)$ If $L$ strongly quasipositive, then $g_4^{top}(L) = g(L) = \frac12(|\sigma(L)| - (m-1))$. 

$(2)$ If $L$ quasipositive, then $g_4^{top}(L) = g_4(L) = \frac12(|\sigma(L)| - (m-1))$. 
\qed
\end{prop}

Let $K$ be a hyperbolic fibred knot in a closed, connected, orientable $3$-manifold $W$ with monodromy $f$. It was observed in \cite[Theorem 4.1]{HKM2} that if the fractional Dehn twist coefficient of $f$ is at least $1$ in absolute value, then work of Roberts \cite{Ro} as interpreted in \cite{HKM2} implies that $W$ is CTF. In a recent preprint \cite{BHu}, Boyer and Hu show that the same condition implies that $W$ is LO. Since the fractional Dehn twist coefficient of $f^n$ is $n$ times the fractional Dehn twist of $f$, and since the fractional Dehn twist coefficient of a hyperbolic fibred strongly quasipositive knot $K$ is non-zero, $\Sigma_n(K)$ is CTF and LO for large $n$.

\begin{problem}
{\rm If $K$ is a fibred strongly quasipositive knot (not necessarily hyperbolic), show that $\Sigma_n(K)$ is LO for $n \gg 0$. }
\end{problem}
 
Of course, given the conjectures described above and the conclusions of Corollary \ref{cor: monic sqp knots}, we expect that in the fibred case, $\Sigma_n(K)$ is CTF and LO for $n \geq 6$.  

\subsection{Strongly quasipositive $2$-bridge knots} 
In the special case that $K$ is a strongly quasipositive $2$-bridge knot, it is interesting to compare what is known about the LO nature of $\Sigma_n(K)$ in comparison with what we've deduced about its L-space nature. (Ying Hu has shown that if $K$ is the $p/q$ $2$-bridge knot where $q \equiv 3$ (mod $4$), then $\Sigma_n(K)$ is LO for $n \gg 1$. See \cite[Theorem 4.3]{Hu}. More generally, Gordon has shown that the same statement holds for $2$-bridge knots with non-zero signatures \cite{Go}.) The CTF status of these manifolds appears to be unknown in general at the present. 

\subsubsection{Genus one $2$-bridge knots} 
Consider the genus one $2$-bridge knots $K = K_{[2k, -2l]}$ where $k > 0$ and $l \ne 0$ are integers (cf. \S \ref{sec: $2$-bridge genus 1}). Since $\Sigma_2(K)$ is a lens space, it is an L-space, has non-left-orderable fundamental group, and supports no co-oriented taut foliation. More generally, it is known that Corollary \ref{cor: 2bridge} holds with ``(not) being an L-space" replaced by ``(does not have) has non-left-orderable fundamental group" (\cite{DPT},  \cite{Tra}). Thus if $l <  0$, $\Sigma_n(K_{[2k, -2l]})$ is not LO for $n \geq 2$ and when $l > 0$ it is LO for $n > 2\pi/\textup{arccos}(1 - 1/2kl)$. Further, in the case that $l > 0$, it is known that $\Sigma_n(K_{[2k, -2l]})$ is an L-space and not LO if $n = 3$ (\cite{Pe} and \cite{DPT} respectively), $n = 4$ (\cite{Te1}, respectively \cite{GLid1}), and $n = 5$ (Hori (Remark \ref{rem: genus one 2-bridge}(2)), respectively \cite{Ba}). 

When $k = 2$ and $l = 1$, $K$ is the knot $5_2$. We discussed what is known about the L-space nature of $\Sigma_n(K)$ in Remark \ref{rem: genus one 2-bridge}(2) and in particular, $\Sigma_n(K)$ is an L-space if and only if $2 \leq n \leq 5$. Marc Culler has used a computer program of Nathan Dunfield to test the left-orderability of $\pi_1(\Sigma_n(K))$ for $n = 6, 7$, and $8$, though the result was inconclusive suggesting that for those values of $n$, $\pi_1(\Sigma_n(K))$ is left-orderable. Finally, we note that the left-orderability of $\pi_1(\Sigma_n(K))$ for $n \ge 9$ uses the fact that these groups have non-abelian representations into $SL(2,\mathbb R)$. In a private communication, Culler and Dunfield have shown that there are no such representations when $n = 6, 7$ or $8$. 

\subsubsection{The $2$-bridge knots $K(k, m)$} \label{subsubsec: Tran}
Next consider $K = K(k, m)$ (cf. \S \ref{sec: 2-bridge}). Again, the fact that $\Sigma_2(K)$ is a lens space implies that it is an L-space, has non-left-orderable fundamental group, and supports no co-oriented taut foliation. Ba has shown that for all $k$, $\Sigma_3(K(k,2))$ is not LO (\cite{Ba}), which should be compared with Corollary \ref{cor: br covers 2bridge 2}(2). It is expected that for $n \geq 4$ or for $n = 3$ when $m \geq 3$, $\Sigma_n(K(k,2))$ is LO. Tran has shown that $\pi_1(\Sigma_n(K(k,m)))$ is left-orderable for $n$ sufficiently large (\cite[Theorem 2(b)]{Tra}), however, the lower bound on $n$ that he obtains for this to happen is a function of $k$ and $m$ which tends to $\infty$ with either $k$ or $m$. (Note that Tran's ``$n$" is our ``$m$", and his ``$m$" is our ``$k-1$".)

\begin{problem}
{\rm Is $\pi_1(\Sigma_n(K(k,m)))$ left-orderable for the values of $n \geq 4$ when $m = 2$ and $n \geq 3$ when $m \geq 3$ (cf. Corollary \ref{cor: br covers 2bridge 2})?}
\end{problem}

\end{document}